\documentclass[11pt]{amsart}
\usepackage{amsmath,amsfonts,amssymb,amsthm,epsfig,enumitem}
\usepackage{color, bm}
\usepackage{comment}
\usepackage[colorlinks=true, pdfstartview=FitV, linkcolor=blue, citecolor=blue, urlcolor=blue, bookmarksdepth=2]{hyperref}
\usepackage[all]{hypcap}
\usepackage{tikz}

\usepackage[margin=1in]{geometry}

\numberwithin{equation}{section}
\newtheorem{Theorem}[equation]{Theorem}
\newtheorem{Proposition}[equation]{Proposition} 
\newtheorem{Lemma}[equation]{Lemma}

\newtheorem{Corollary}[equation]{Corollary}

\theoremstyle{definition}

\newtheorem{Question}[equation]{Question}

\newcommand{\arxiv}[1]{\href{http://arxiv.org/abs/#1}{\tt arXiv:\nolinkurl{#1}}}

\newcommand{\mF}{\mathfrak{F}}
\newcommand{\oOmega}{{\overline \Omega}}

\newcommand{\cT}{\mathcal{T}}
\newcommand{\bl}{{\bm \lambda}}
\newcommand{\bmu}{{\bm \mu}}

\newcommand{\wt}{\operatorname{wt}}

\newcommand{\bz}{\Bbb{Z}}
\newcommand{\ba}{\Bbb{A}}
\newcommand{\bc}{\Bbb{C}}
\newcommand{\br}{\Bbb{R}}
\newcommand{\bn}{\Bbb{N}}
\newcommand{\g}{\mathfrak{g}}

\newcommand{\asl}{\widehat{\mathfrak{{sl}}}}

\newcommand{\hook}{\operatorname{hook}}
\newcommand{\arm}{\operatorname{arm}}
\newcommand{\leg}{\operatorname{leg}}

\newcommand{\cP}{\mathcal{P}}

\newcommand{\Irr}{\operatorname{Irr}}

\newcommand{\Quot}{{\bf Quot}}

\newcommand{\bfv}{{\bf v}}
\newcommand{\bfw}{{\bf w}}

\newcommand{\im}{\operatorname{im}}
\newcommand{\ii}{{\bar \imath}}

\newcommand{\mL}{\mathfrak{L}}
\newcommand{\mM}{\mathfrak{M}}
\newcommand{\ol}{\overline}

\newcommand{\tr}{\text{tr}}

\newcommand{\mE}{\mathcal{E}}

\newcommand{\MM}{\mathcal{M}}
\newcommand{\bp}{\mathbb{P}}
\newcommand{\cO}{\mathcal{O}}
\newcommand{\eps}{\varepsilon}

\newcommand{\GL}{\mathbf{GL}}

\newcommand{\bx}{{\bm \xi}}

\newlength{\arrowsize}  
\pgfarrowsdeclare{biggertip}{biggertip}{  
  \setlength{\arrowsize}{0.4pt}  
  \addtolength{\arrowsize}{.5\pgflinewidth}  
  \pgfarrowsrightextend{0}  
  \pgfarrowsleftextend{-5\arrowsize}  
}{  
  \setlength{\arrowsize}{1pt}  
  \addtolength{\arrowsize}{.5\pgflinewidth}  
  \pgfpathmoveto{\pgfpoint{-5\arrowsize}{4\arrowsize}}  
  \pgfpathlineto{\pgfpointorigin}  
  \pgfpathlineto{\pgfpoint{-5\arrowsize}{-4\arrowsize}}  
  \pgfusepathqstroke  
}  

\theoremstyle{definition}
\newtheorem{example}[equation]{Example}

\newtheorem{remark}[equation]{Remark}
\newenvironment{Remark}[1][]{\begin{remark}[#1]\pushQED{\qed}}{\popQED \end{remark}}
\newtheorem{defn}[equation]{Definition}
\newenvironment{Definition}[1][]{\begin{defn}[#1]\pushQED{\qed}}{\popQED \end{defn}}

\begin{document}

\title[Crystals and torus actions]{Combinatorial realizations of
  crystals via torus actions on quiver varieties}

\subjclass[2010]{%
05E10, 
17B10
}

\author{Steven V Sam}
\address{Department of Mathematics, University of California, Berkeley, CA}
\email{svs@math.berkeley.edu}
\author{Peter Tingley}
\address{Department of Mathematics and Statistics, Loyola University, Chicago, IL}
\email{ptingley@luc.edu}

\begin{abstract}
Let $V(\lambda)$ be a highest weight representation of a symmetric Kac--Moody algebra, and let $B(\lambda)$ be its crystal.  There is a geometric realization of $B(\lambda)$ using Nakajima's quiver varieties.  In many particular cases one can also realize $B(\lambda)$ by elementary combinatorial methods. Here we study a general method of extracting combinatorial realizations from the geometric picture: we use Morse theory to index the irreducible components by connected components of the subvariety of fixed points for a certain torus action.  We then discuss the case of $\asl_n$, where the fixed point components are just points, and are naturally indexed by multi-partitions. There is some choice in our construction, leading to a family of combinatorial realizations for each highest weight crystal. In the case of $B(\Lambda_0)$ we recover a family of realizations which was recently constructed by Fayers.  This gives a more conceptual proof of Fayers' result as well as a generalization to higher level crystals. We also discuss a relationship with Nakajima's monomial crystal. 
\end{abstract}

\maketitle

\setcounter{tocdepth}{1} 

\tableofcontents 

\section{Introduction}
Kashiwara introduced a combinatorial object (a set along with certain operators) called a crystal associated to each irreducible highest weight representation of a symmetrizable Kac--Moody algebras $\g$, which encodes various information about the representation. 
Kashiwara's theory makes heavy use of the quantized universal enveloping algebra associated with $\g$, but the crystals themselves can often be realized by other means. For instance, in the case of the fundamental crystal $B(\Lambda_0)$ for the affine Kac--Moody algebras $\asl_n$, Misra and Miwa \cite{MM:1990} give a realization based on certain partitions.  Recently Fayers \cite{Fayers:2009}, building on work of Berg \cite{Berg:2010}, found an uncountable family of modifications to the Misra--Miwa realization, and hence many seemingly different realizations of the same crystal. 

Fayers' construction is purely combinatorial, and the motivation for the current work was to find a conceptual explanation of the existence of this family. We achieve this using Nakajima's quiver varieties. Our construction also allows us to generalize Fayers' results to give families of realizations of $B(\Lambda)$ in terms of multi-partitions for any integrable highest weight $\asl_n$-module $V(\Lambda)$. Most of our construction is actually carried out in the generality of symmetric Kac--Moody algebras, although the end result is less combinatorial in other cases. 

For the moment fix a symmetric Kac--Moody algebra $\g$ and a highest weight representation $V(\lambda)$.
Nakajima's quiver varieties give a geometric way to understand the corresponding crystal $B(\lambda)$. 
One defines a variety $\mathfrak{L}(\bfv,W)$ for each choice of a (graded) vector space $W$ and a dimension vector $\bfv$. Choosing a specific $W$ whose dimension is determined by $\lambda$, the vertices of $B(\lambda)$ are indexed by
\[
\coprod_{\bfv} \Irr \mathfrak{L}(\bfv, W),
\]
the union of the set of irreducible components of $\mathfrak{L}(\bfv, W)$ as $\bfv$ varies. 

There is a natural action of a torus $T$ on each $\mathfrak{L}(\bfv, W)$. 
We denote the subvariety of torus fixed points by $\mathfrak{F}(\bfv, W)$. In fact, the torus action extends to a larger smooth variety $\mathfrak{M}(\bfv, W)$, and there are no new fixed points in $\mathfrak{M}(\bfv, W)$, which implies that $\mathfrak{F}(\bfv, W)$ is a smooth subvariety of $\mathfrak{L}(\bfv, W)$. For a generic 1-parameter subgroup $\iota \colon \bc^* \rightarrow T$ (when $\g$ is of infinite type, $\iota$ has to be interpreted as a certain limit of 1-parameter subgroups), there is a map
\begin{equation*}
\begin{aligned}
M_\iota \colon  \coprod_\bfv \Irr \mathfrak{L}(\bfv, W) & \rightarrow \coprod_\bfv  \Irr \mathfrak{F}(\bfv, W) \\
Z & \mapsto \text{Component containing } \lim_{t \rightarrow \infty} \iota(t) \cdot x \text{ for $x \in Z$ generic.}
\end{aligned}
\end{equation*}
For many $\iota$ this map is 1-1. In these cases one can transport the crystal structure on $\coprod_\bfv \Irr \mathfrak{L}(\bfv, W)$ to a crystal structure on its image in $\coprod_\bfv \Irr \mathfrak{F}(\bfv, W)$. 
This latter set can sometimes be described combinatorially, which is how we get combinatorics out of the geometry. In fact, this  idea has previously been used by Nakajima \cite{Nakajima:2001}, although here we use a larger torus $T$, so we are able to see some combinatorics which was not visible in that work.

In the current paper we mainly consider the case of $\asl_n$. Here the construction is particularly nice because each fixed point variety $\mathfrak{F}(\bfv, W) $ is a finite collection of points. Even better, these points are naturally indexed by tuples of partitions. By taking various choices of $\iota$, we get a large family of realizations for each highest weight crystal $B(\Lambda)$ where the vertices are certain multi-partitions. For many of these choices we give a simple characterization of the image of $M_\iota$ and a combinatorial description of the corresponding crystal operators on multi-partitions (see Theorems~\ref{thm:to-partitions} and~\ref{thm:irrationalcrystal}). This gives a large family of combinatorial realizations of each highest weight crystal $B(\Lambda)$ for $\asl_n$. 
In the case $\Lambda=\Lambda_0$ we obtain exactly the realizations found by Fayers.

In \S\ref{sec:monomial}, we give an application to the monomial crystal of Nakajima \cite[\S3]{Nakajima:2003}, as generalized by Kashiwara \cite[\S4]{Kashiwara:2001}. Specifically, we show that, for each of the realizations of $B(\Lambda)$ discussed above,   there is a map to a particular instance of the monomial crystal. One consequence of this is that certain instances of the monomial which have not previously been studied do in fact realize the crystals $B(\Lambda)$.

Before beginning, we would like to mention related work of Savage \cite{Savage:2006} and Frenkel--Savage \cite{frenkelsavage} discussing how to extract various combinatorial realizations of crystals from Nakajima's quiver varieties. One special case is $B(\Lambda_0)$ for $\asl_n$, where Savage demonstrates the appearance of the Misra--Miwa realization. Savage parameterizes the irreducible components of the varieties as conormal bundles of various orbits of representations of an undoubled cyclic quiver. By instead viewing the irreducible components as parameterized by torus fixed points, we gain the freedom to flow towards those fixed points in many different ways, thereby seeing a whole family of realizations where Savage only saw one. Even earlier work which uses ideas similar to the ones in this paper can be found in \cite{Nakajima:1994b}. We also point to \cite{fujiiminabe} for a survey of topics including affine type A quiver varieties, quot schemes, and torus actions, which discusses many of the tools used in this paper.

\subsection*{Acknowledgments}
We thank Pavel Etingof, Matthew Fayers, Monica Vazirani and Ben Webster for interesting discussions. We also thank Hiraku Nakajima, Alistair Savage, Ben Webster, and two anonymous referees for helpful comments on earlier versions of this paper. The first author was supported by an NDSEG fellowship and a Miller research fellowship. The second author was supported by NSF grants DMS-0902649, DMS-1162385 and DMS-1265555.

\section{Notation}

We give a table of important notation and where it is first used. We only include notation that is used in multiple sections. 

~

\begin{tabular}{l|l|l}
  Notation & Description & First used\\ \hline
  $\mM(\bfv,W), \mL(\bfv,W)$ & Quiver varieties  &
  \S\ref{ssec:Nak-varieties}\\
    $\lambda$; $\omega_i$ & Highest weights; fundamental weights in general & \S\ref{ss:crystals} \\
  $B(\lambda), B(\infty)$ & Crystals & \S\ref{ss:crystals} \\
   $e_i$, $f_i$ & Crystal operators & \S\ref{ss:crystals} and \S\ref{ss:quiver-crystal} \\ 
$\Gamma, I, A, Q, H$ & Quiver notation & \S\ref{ssec:Nak-varieties}\\
   $\Irr X$ & The irreducible components of $X$ & \S\ref{ss:quiver-crystal} \\
 $x_\ii$, $\ol{x}_\ii$, $s$, $t$ & Operators on quiver representations &
  \S\ref{ss:quiver-crystal} \\
  $\cT$, $T_\Omega$, $T_W$, $T_s$ & Tori in general framework & \S\ref{ssec:torus} \\
  $\mF(\bfv,W)$ & $\cT$-fixed points of $\mL(\bfv,W)$ &
  \S\ref{sec:framework} \\
  $F_\iota(\bfv, W)$ & $\bc^*$-fixed points for a map $\iota \colon \bc^* \rightarrow \cT$ & \S\ref{sec:framework} \\
  $M_\iota$, $M_\bx$ & Maps from $\Irr \mL(\bfv,W)$ to $\Irr \mF(\bfv,W)$ & \S\ref{sec:framework} and \S\ref{sec:imageMiota} \\
  $\bl$, $p$ & Multi-partition and coloring function & \S\ref{partitions-section} \\
  $\bar{c}(b)$, $c(\bl)$ &The color of a box $b \in \bl$; the content of $\bl$ & \S\ref{partitions-section} \\
  $\arm$, $\leg$ & Arm and leg statistics & \S\ref{partitions-section} \\
$A_\ii, R_\ii$ & Addable and removable $\ii$-nodes & Definition~\ref{AR-def}\\
  $\bx$, $\xi_\Omega$, $\xi_{\ol{\Omega}}$, $\xi_i$ & Slope
  datum & Definition~\ref{defn:slopedatum} \\
    $h^\bx$ & The height function associated to $\bx$ & Definition~\ref{defn:height} \\
  integral, general, aligned  &  Conditions on slope datum &
  Definition~\ref{defn:height} \\
  $\bx$-regular, $\bx$-illegal & Conditions on multi-partitions & Definition~\ref{def:multiillegal} \\
$T$, $(t_\Omega, t_{\ol{\Omega}}, t_1, \dots, t_\ell)$ & $\asl_n$ specific torus and its coordinates & \S\ref{ss:torus-combinatorics} \\
  $\Lambda$; $\Lambda_\ii$ & Highest weights; fundamental weights for $\asl_n$ & \S\ref{ss:torus-combinatorics} \\
$p_\bl$ & $T$ fixed point in $\mL(\bfv, W)$ corresponding to $\bl$ & Figure~\ref{fig:fixed} \\
$\mE_\bl$ & Affine space locally containing $\mM(\bfv, W)$ near $p_\bl$ & \S\ref{ss:coordinates} \\
$T_\bx$ & 1-parameter subgroup attached to slope datum & \S\ref{sec:imageMiota}\\
$e_{\ii}^\bx, f_\ii^\bx$ & Crystal operators attached to slope datum & \S\ref{ss:comb-cryst}
\end{tabular}

\section{Background}

\subsection{Crystals} \label{ss:crystals}
We refer the reader to \cite{Kashiwara:1995} or \cite{Hong&Kang:2000} for more on this rich subject. Here we simply fix notation and state the properties of crystals we need. We only consider certain explicit realizations of crystals so do not need to discuss the general theory. 

A crystal for a symmetric Kac--Moody algebra $\mathfrak{g}$ consists of a set $B$ along with operators $e_{i}, f_{i} \colon B \rightarrow B \cup \{ 0 \}$ for each $i \in I$, which satisfy various axioms. There is a crystal $B(\lambda)$ associated to each integral highest weight $\lambda$, which records certain combinatorial information about the irreducible highest weight representation $V(\lambda)$. Often the definition of a crystal includes three functions $\wt, \varphi, \varepsilon \colon B \rightarrow P$, where $P$ is the weight lattice. In the case of crystals of $B(\lambda)$, these functions can be recovered (up to shifts in null directions) from knowledge of the $e_{i}$ and $f_{i}$, so we will not count them as part of the data. 

When $\lambda-\mu$ is a dominant integral weight, there is a unique embedding $B(\mu) \hookrightarrow B(\lambda)$ that commutes with all the operators $e_i$ (although not with the operators $f_i$). In this way $\{B(\lambda)\}$ becomes a directed system. This system has a limit called the infinity crystal, which we denote by $B(\infty)$.

\subsection{Quiver varieties} \label{ssec:Nak-varieties}

Here we review the quiver varieties originally constructed by Lusztig \cite{Lusztig:1991} and Nakajima \cite{Nakajima:1994, Nakajima:1998}. We work only with quiver varieties over the field $\bc$ of complex numbers.

Fix a graph $\Gamma=(I, H)$, and let $Q = (I, A)$ be its doubled quiver (i.e., the directed graph with two arrows for each edge in $\Gamma$, one in each direction). Let $t(a)$ and $h(a)$ denote the tail and head of the arrow $a$. Choose an orientation of $Q$, which is a subset $\Omega \subset A$ that contains exactly one arrow for each edge in $\Gamma$. Define 
\begin{equation*}
\epsilon(a) = \begin{cases}
1 & \text{if } a \in \Omega \\
-1 & \text{otherwise.}
\end{cases}
\end{equation*}
Define an involution $a \mapsto \ol{a}$ on $A$ by setting $\overline{a}$ to be the arrow associated to the same edge as $a$, but in the reverse direction.

\begin{figure}
\setlength{\unitlength}{0.5cm}
\vspace{1cm}
\hspace{1cm}

\begin{picture}(5,-5)(-5,0)

\put(-15.15,-0.8){\tiny $1$}
\put(-13.15,-0.8){\tiny $2$}

\put(-7.5,-0.8){\tiny $n-2$}
\put(-5.5,-0.8){\tiny $n-1$}

\put(-10.15,2){\tiny $0$}

\put(-15,0){\circle*{0.4}}
\put(-13,0){\circle*{0.4}}

\put(-7,0){\circle*{0.4}}
\put(-5,0){\circle*{0.4}}

\put(-10,1.2){\circle*{0.4}}

\put(-10,1.2){\line(4,-1){5}}
\put(-10,1.2){\line(-4,-1){5}}

\put(-10.4,0){$\ldots$}

\put(-14.8,0){\line(1, 0){1.6}}
\put(-12.8,0){\line(1, 0){1.6}}
\put(-8.8,0){\line(1, 0){1.6}}
\put(-6.8,0){\line(1, 0){1.6}}

{\color{red}
\put(1.3,0.2){\vector(1, 0){1.6}}
\put(3.3,0.2){\vector(1, 0){1.6}}
\put(7.3,0.2){\vector(1, 0){1.6}}
\put(9.3,0.2){\vector(1, 0){1.6}}
}

{\color{black}
\put(10.7,-0.2){\vector(-1, 0){1.6}}
\put(8.7,-0.2){\vector(-1, 0){1.6}}
\put(4.7,-0.2){\vector(-1, 0){1.6}}
\put(2.7,-0.2){\vector(-1, 0){1.6}}
}

{\color{black}
\put(1, 0.25){\vector(4,1){4.6}} }
{\color{red}
\put(5.4, 1.15){\vector(-4,-1){4.6}}
}

{\color{black}
\put(5.4, 1.35){\vector(4,-1){4.2}} }
{\color{red}
\put(9.6, 0.05){\vector(-4,1){4.2}}
}

\put(0,0){\circle*{0.4}}
\put(2,0){\circle*{0.4}}

\put(8,0){\circle*{0.4}}
\put(10,0){\circle*{0.4}}

\put(5,1.3){\circle*{0.4}}

\put(4.3, -0.1){$\cdots$}

\put(-0.15,-0.8){\tiny $1$}
\put(1.85,-0.8){\tiny $2$}

\put(7.5,-0.8){\tiny $n-2$}
\put(9.5,-0.8){\tiny $n-1$}

\put(4.85,2){\tiny $0$}
\end{picture}

\caption{\label{pic:QV} The $\asl_n$ Dynkin diagram and the associated doubled quiver $Q$. In $Q$, the negatively oriented arrows (the ones of the form $i\to i+1$) are colored red. These are the arrows for which $\epsilon(a)=-1$. }
\end{figure}
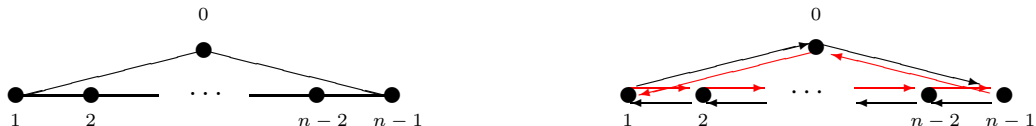

Let $V$ be an $I$-graded vector space of dimension $\bfv$. Let $E(V)$ be the space of all representations of the path algebra of $Q$ on $V$ (that is, an element $x \in E(V)$ consists of a choice of linear map $x_a \colon V_{t(a)} \rightarrow V_{h(a)}$ for each arrow $a$). 
$E(V)$ is a symplectic vector space, where the form $\langle \cdot, \cdot \rangle$ is defined by
\begin{equation*}
\langle x, x'\rangle = \sum_{a \in A} \epsilon(a) \tr( x'_{\bar a} x_a ).
\end{equation*}
Here $\tr$ means trace. There is a natural action of $\GL(V) = \prod_i \text{GL} (V_i)$ on $E(V)$, and the moment map $\mu \colon E(V) \to \mathfrak{gl}(V) = \prod_i \mathfrak{gl}(V_i)$ for this action is given by 
\begin{align*}
\mu(x) = \bigoplus_{i \in I} \left[ \sum_{a: t(a)=i} \epsilon(a) x_{\bar a} x_a  \right].
\end{align*}
Here we are identifying $\mathfrak{gl}(V)$ with its dual space by the trace form.

A point $x \in E(V)$ is called {\bf nilpotent} if, for some $N$ and all paths $a_N \cdots a_2 a_1$ in $Q$ of length $N$, $x_{a_N} \circ \cdots \circ x_{a_2} \circ x_{a_1} =0$ as a map from $V_{t(a_1)}$ to $V_{h(a_N)}$. 
The {\bf Lusztig quiver variety} $\Lambda(V)$ is the subvariety of $E(V)$ consisting of nilpotent points which also satisfy the preprojective relations $\mu(x)= (0)$.
As discussed in \cite{Lusztig:1991}, $\Lambda(V)$ is a Lagrangian subvariety of $E(V)$. 

Now choose another $I$-graded vector space $W = \bigoplus_i W_i$. Let $E(V,W)$ be the space of all triples $(x,s,t)$ where $x \in E(V)$, and $s \colon V \rightarrow W$ and $t \colon W \rightarrow V$ are maps of $I$-graded vector spaces. 
$E(V,W)$ has a symplectic form $\langle \cdot, \cdot \rangle$ defined by
\begin{equation*}
\langle (x,s,t), (x',s',t') \rangle = \tr (s' t) - \tr (t' s) + \sum_{a \in A} \epsilon(a) \tr( x'_{\bar a} x_a ).
\end{equation*}
There is a natural action of $\GL(V)$ on $E(V,W)$, and the moment map $\mu \colon E(V,W) \to \mathfrak{gl}(V)$ for this action is given by 
\begin{align*}
\mu(x,s,t) = \bigoplus_{i \in I} \left[ \sum_{a: t(a)=i} \epsilon(a) x_{\bar a} x_a + t_i s_i  \right].
\end{align*}

\begin{Definition}
We call $(x,s,t) \in E(V,W)$ {\bf stable} if $\im(t)$ generates $V$ under the action of $x$. Denote the subset of $E(V,W)$ consisting of stable representations by $E(V,W)^{\rm st}$.
\end{Definition}

The stability condition ensures that the $\GL(V)$ action on $E(V,W)^{\rm st}$ is free, so we can define the following varieties
\begin{equation} \label{eq:qv}
  \begin{split}
    \mM(\bfv,W) &:= \{ (x,s,t) \in E(V,W)^{\rm st} \mid \mu(x,s,t)=0
    \} / \GL(V), \\
    \mL(\bfv,W) &:= \{ [x,s,t] \in \mathfrak{M}(\bfv,W) \mid s=0, \; x
    \text{ is nilpotent} \}.
\end{split}
\end{equation}
Here $\bfv = \dim V$ (as an $I$-graded vector space). Furthermore, it follows from the theory of Marsden--Weinstein reduction \cite{MW} that $\mM(\bfv,W)$ is smooth, and the symplectic form on $E(V,W)$ descends to a symplectic form on $\mathfrak{M}(\bfv,W)$. Alternatively, these constructions can be done in the language of geometric invariant theory, from which one can deduce that $\mL(\bfv, W)$ is a projective algebraic variety (see \cite[\S 3.iii]{Nakajima:1998}).

\begin{Theorem}[{\cite[Theorem 5.8]{Nakajima:1994}}] \label{L:th}
Assume that $Q$ does not have any loops (i.e., edges starting and ending at the same vertex). Then $\mathfrak{L}(\bfv,W)$ is a Lagrangian subvariety of $\mathfrak{M}(\bfv,W)$. In particular, it is equidimensional of dimension $\frac{1}{2}\dim \mM(\bfv,W)$. \qed
\end{Theorem}

\begin{Remark}
The varieties $\mM(\bfv, W)$ and $\mL(\bfv, W)$ constructed using different spaces $V$ with the same graded dimension $\bfv$ are canonically isomorphic, which is why we only record the dimension $\bfv$. Up to isomorphism, the space also only depends on the graded dimension of $W$, but that isomorphism is not canonical, and it is useful to keep track of a fixed vector space $W$. On occasion we will need to refer to a choice of vector space $V$ associate to a point in $\mM(\bfv, W)$, in which case we will refer to the point by $[V,x,s,t]$ instead of just $[x,s,t]$. 
\end{Remark}

\subsection{Crystal structure on quiver varieties} \label{ss:quiver-crystal}

The following construction is due to Kashiwara and Saito \cite{kashiwarasaito, Saito:2002}, although we have rephrased things slightly. 

Given $x \in \Lambda(V)$ and $i \in I$, define 
\[
x_i:=  \sum_{a \colon i \to j} x_a \colon V_i \to \bigoplus_{a: i \to j} V_j \quad \text{ and } \quad {}_ix :=\sum_{a \colon j\to i} \epsilon(a) x_a \colon \bigoplus_{a: j\to i} V_j \rightarrow V_i. 
\]
The condition $\mu(x)=0$ in the definition of $\Lambda(V)$ is equivalent to, for all $i$, ${}_i x \circ x_i =0$.

Each irreducible component of $\Lambda(V)$ is $\GL(V)$-invariant, so we can safely denote the set of irreducible components by $\Irr \Lambda(\bfv)$, only recording the dimension vector ${\bfv}$ of $V$. 
Given $Z \in \Irr \Lambda(\bfv)$ and $i \in I$, the quantities $\dim \im(x_i)$ and $\dim \im({}_i x)$ are semi-continuous functions. Hence we can define 
\[
Z_i^0 = \{ x \in Z \mid \dim \im(x_i) \text{ is maximal and} \dim \im({}_i x) \text{ is maximal}\},
\]
which is an open dense subset of $Z$. Define
\begin{align*}
e_i(Z) &= \ol{\{ x \in \Lambda(\bfv - {\bf 1}_i) \mid x \text{ is isomorphic to a quotient of some } x' \in Z_i^0\}}\\
f_i(Z) &= \ol{\{ x \in \Lambda(\bfv + {\bf 1}_i) \mid x \text{ has a quotient isomorphic to some } x' \in Z_i^0\}}.
\end{align*}
where ${\bf 1}_i$ is the vector with coordinate $1$ in position $i$ and $0$ elsewhere. Then $e_i(Z) \in \Irr \Lambda(\bfv - {\bf 1}_i) \cup \{\varnothing\}$ and $f_i(Z) \in \Irr \Lambda(\bfv + {\bf 1}_i)$, and by \cite[Theorem 5.3.2]{kashiwarasaito}, $\coprod_{\bfv} \Irr \Lambda(\bfv)$ along with the operators $e_i, f_i$ is a realization of the infinity crystal $B(\infty)$.

To realize the highest weight crystals $B(\lambda)$ we must move to Nakajima's quiver varieties. Fix a dominant integral weight $\lambda = \sum_i w_i \omega_i$ expressed as a linear combination of fundamental weights. Fix an $I$-graded vector space $W$ of dimension $\bfw=(w_i)_{i \in I}$. Given an irreducible component $Z \in \Irr \Lambda(\bfv)$, define
\[
Z^W = \{ [x,0,t] \in \mL(\bfv,W) \mid x \in Z,\ (x,0,t) \in E(V,W)^{\rm st} \}.
\]
Then $Z^W$ is either empty or an irreducible component of $\mL(\bfv, W)$. Define the map
\begin{align*}
G \colon  \coprod_{\bfv} \Irr \Lambda(\bfv) & \rightarrow \coprod_{\bfv} \Irr \mL(\bfv, W) \cup \{ \emptyset \} \\
Z & \mapsto Z^W
\end{align*}
and let $S \subset \coprod_{\bfv} \Irr \Lambda(\bfv)$ be the set of those $Z$ such that $Z^W \neq \emptyset$. Then $G$ is a bijection between $S$ and $\coprod_{\bfv} \Irr \mL(\bfv, W)$, and we denote the inverse bijection by $G^{-1}$. By \cite[\S 4.6]{Saito:2002}, $\coprod_{\bfv} \Irr \mL(\bfv, W)$ along with the crystal operators $G e_i G^{-1}$ and $G f_i G^{-1}$ is a realization of $B(\lambda)$. 

\subsection{Torus actions} \label{ssec:torus} 
One of the main tools in the current paper is a large torus acting on each variety $\mL(\bfv, W)$. This torus is the product of three smaller tori, which we now define. 

Let $T_\Omega \simeq (\bc^*)^{\Omega}$. This acts on $\mathfrak{M}(\bfv,W)$, where ${\bf z} = (z_a)_{a \in \Omega}$ acts by
\begin{equation*}
({\bf z} \cdot x)_a = \begin{cases}
z_a x_a  & \text{if } a \in \Omega \\
z_{\bar a}^{-1} x_{a} & \text{otherwise},
\end{cases}
\end{equation*}
and fixes $s$ and $t$.

Fix a maximal torus $T_W \cong (\bc^*)^{\dim W}$ in $\GL(W)$ compatible with the $I$-grading. Then $T_W$ acts on $\mathfrak{M}(\bfv,W)$ by
\begin{equation*}
M \cdot [x,s,t] = [x, Ms, t M^{-1}]. 
\end{equation*}

Finally, consider the one-dimensional torus $T_s \cong \bc^*$ which acts on $\mM(\bfv, W)$ by 
\begin{equation*}
(w \cdot x)_a = 
\begin{cases}
w x_a & \text{if } a \in \Omega \\
x_a & \text{otherwise}, 
\end{cases}
\quad
w \cdot s = w s, \quad w \cdot t = t.
\end{equation*}

It is clear that all these actions preserve $\mL(\bfv, W) \subset \mM(\bfv, W)$. 
Let $\cT = T_\Omega \times T_W \times T_s$.

\begin{Remark} \label{rem:W-form} 
The torus $T_I := (\bc^*)^I$ naturally embeds in $\cT$, and the induced action of $T_I$ on $\mathfrak{M}(\bfv,W)$ is trivial. Thus we actually have an action of the quotient $\cT/T_I$. If $\Gamma$ is a tree, one can see that $\cT/T_I  \cong (T_W \times T_s)/ (T_I \cap (T_W \times T_s))$, so the orbit of any point under $\cT$ is the same as the orbit under $T_W \times T_s$. Thus $T_\Omega$ only contributes non-trivially when $\Gamma$ has at least one cycle. 
\end{Remark}

\section{A framework for extracting combinatorics} \label{sec:framework}

Let $\cT = T_\Omega \times T_W \times T_s$, which acts on $\mathfrak{M}(\bfv,W)$ and $\mathfrak{L}(\bfv,W)$ as in \S\ref{ssec:torus}.  Consider a 1-parameter subgroup $\iota \colon \bc^* \hookrightarrow \cT$, and denote the induced $\bc^*$ action on $\mathfrak{M}(\bfv,W)$ by $T_\iota$. Consider $\pi_s \circ \iota \colon \bc^* \rightarrow \bc^*$, where $\pi_s$ is projection onto $T_s$. Define $\wt(T_\iota)$ to be the weight of $\pi_s \circ \iota$. The following is clear from the definitions:

\begin{Lemma}
$T_\iota$ acts with weight $\wt(T_\iota)$ on the symplectic form from \S\ref{ssec:Nak-varieties}.
\qed
\end{Lemma}

Let $F_\iota(\bfv, W)$ be the variety of fixed points of $\mathfrak{M}(\bfv,W)$ under the action of $T_\iota$.  By \cite[Lemma 5.11.1]{CG:1997}, each connected component of $F_\iota(\bfv,W)$ is a smooth subvariety of $\mathfrak{M}(\bfv,W)$. For each connected component $C$ of $F_\iota(\bfv,W)$, let $A_C$ be the subvariety
\begin{equation*} 
A_C := \{ x \in \mathfrak{M}(\bfv,W) \mid \lim_{t \rightarrow \infty} t \cdot x \in C \}. 
\end{equation*}
By \cite[Theorem 4.1]{BB}, $A_C$ is an affine bundle over $C$, so is smooth and irreducible. 

Now define a map 
\begin{equation*}
\begin{aligned}
M_\iota \colon \Irr \mL(\bfv,W) & \rightarrow \Irr F_\iota(\bfv,W) \\
Z & \mapsto \text{Component $C$ such that $A_C \cap Z$ is dense in $Z$.}
\end{aligned}
\end{equation*}
Alternatively, we can think of $C$ as the irreducible component that contains $\lim_{t \to \infty} t \cdot x$ for generic $x \in Z$. Since $\mL(\bfv,W)$ is projective, and the set of points flowing to a given fixed point component is always algebraic, this map is well-defined.

\begin{Proposition} \label{prop:1-1}
If $\wt T_\iota>0$ then $M_\iota$ is injective. 
\end{Proposition}

\begin{proof} 
Pick $C \in \Irr F_\iota(\bfv,W)$ and $x \in C$. The symplectic form $\langle \cdot, \cdot \rangle$ on the tangent space ${\rm T}_x \mathfrak{M}(\bfv,W)$ is non-degenerate and $T_\iota$ acts with positive weight, so the tangent vectors in $A_C$ at $x$ form an isotropic subspace of ${\rm T}_x \mathfrak{M}(\bfv,W)$. Hence $\dim A_C \leq \frac{1}{2} \dim \mathfrak{M}(\bfv,W) = \dim \mathfrak{L}(\bfv,W)$. But $A_C$ is irreducible and $\mathfrak{L}(\bfv, W)$ is equidimensional, so $A_C$ cannot contain a dense subset of two distinct irreducible components.
\end{proof}

\begin{Remark}
Proposition \ref{prop:1-1} shows that, for any $\iota$ of positive weight, we can transport the crystal structure on $\coprod_\bfv \Irr \mL(\bfv,W)$ to a crystal structure on some subset of $\coprod_\bfv \Irr F_\iota(\bfv, W)$, and hence this gives a realization of $B(\Lambda)$.
\end{Remark}

In fact, we often want to work with the varieties $\mathfrak{F}(\bfv, W)$ of fixed points with respect to the whole torus $\cT$. For any given $\bfv$, a generic $\iota$ will satisfy $F_\iota(\bfv, W)= \mathfrak{F}(\bfv, W)$. However, we need to consider $\coprod_\bfv F_\iota(\bfv,W)$, and it can happen (in infinite type) that no $\iota$ satisfies $F_\iota(\bfv, W) = \mathfrak{F}(\bfv, W)$ for all $\bfv$ simultaneously. In these cases, we actually work with a collection of 1-parameter subgroups $\iota^{(N)}$ for all $N \in \bn$ with the following properties:
\begin{enumerate}
\item[(F1)] For any fixed $\bfv$ and all sufficiently large $N$, the fixed points of $T_{\iota^{(N)}}$ acting on $\mathfrak{M}(\bfv,W)$ are exactly the fixed points of $\cT$ acting on $\mathfrak{M}(\bfv,W)$, i.e., we have $F_{\iota^{(N)}}(\bfv,W)= \mathfrak{F}(\bfv, W)$.

\item[(F2)] For any fixed $\bfv$, the map $M_{\iota^{(N)}} \colon \Irr \mathfrak{M}(\bfv,W) \rightarrow \Irr F_{\iota^{(N)}}(\bfv, W)$ stabilizes for large enough $N$ (recalling that for large $N$ we have $\Irr F_{\iota^{(N)}}(\bfv, W)=  \Irr \mathfrak{F}(\bfv, W)$).

\item[(F3)] $\wt \iota^{(N)}$ is positive for all $N$.
\end{enumerate}
The same arguments as above show that such a family $\iota = \{ \iota^{(N)} \}$ gives an injective map 
\[
M_\iota \colon \coprod_\bfv \Irr \mL(\bfv,W) \to \coprod_\bfv \Irr \mathfrak{F}(\bfv,W).
\]

\begin{Remark}
We will be interested in collections $\iota = \{\iota^{(N)} \}_{N \in \bn}$ defined as follows: Fix an isomorphism $\cT \rightarrow (\bc^*)^k$, where $k={|\Omega|+ \dim_\bc(W)+1}$. Let $\pi_1, \ldots, \pi_k$ be the projections onto the various factors, and assume that $\pi_k = \pi_s$ as above. Fix real numbers $\xi_1, \ldots, \xi_{k-1}$, and choose a sequence $\iota^{(N)}$ such that $\wt(\pi_s \circ \iota^{(N)})$ is always positive, and for each $1 \leq r \leq k-1$,
\begin{equation*}
\lim_{N \rightarrow \infty} \frac{\wt(\pi_r \circ \iota^{(N)})}{\wt(\pi_s \circ \iota^{(N)})} = \xi_r. 
\end{equation*}
It is clear that, for generic $(\xi_1, \ldots, \xi_{k-1})$, the family $\iota^{(N)}$ has the required properties.
\end{Remark}

\begin{Remark}
We will see that, if $\Gamma= {\rm A}_n $ or ${\rm A}^{(1)}_n$, then all $\mathfrak{F}(\bfv, W)$ are finite sets of points. In fact, these are the only cases with connected $\Gamma$ where this happens. To see this, consider type ${\rm D}_4$, where the branch node is labeled $2$. For $W$ of dimension $(0,1,0,0)$, one can explicitly calculate that the irreducible component of $\mL(\bfv, W)$ corresponding to $f_2 f_1 f_3 f_4 f_2 v_+\in B(\infty)$ is isomorphic to ${\Bbb P}^1$, and that this whole component is fixed under the torus action. Any connected simply-laced Dynkin diagram other than ${\rm A}_n $ or ${\rm A}^{(1)}_n$ contains ${\rm D}_4$ as a subdiagram, so there is a fixed point component isomorphic to ${\Bbb P}^1$. 
\end{Remark}

\section{${\asl}_n$ specific background and definitions} \label{section:asl}

There are many reasons to expect the construction of \S\ref{sec:framework} to be particularly interesting for $\asl_n$. First, these are the simplest Dynkin diagrams with a cycle, and hence the simplest examples where $T_\Omega$ contributes (see Remark \ref{rem:W-form}). Second, in these cases the fixed point sets of $\cT$ are isolated; in fact they are naturally indexed by multi-partitions, so our methods lead to realizations of crystals based on these combinatorial objects. The rest of this paper will be devoted to this special case, and in this section we begin by collecting various $\asl_n$ specific background and introducing some specialized notation. From now on all quiver varieties are for the $\asl_n$ quiver with the cyclic orientation $\Omega$ shown in Figure~\ref{pic:QV}.

\subsection{Partitions and multi-partitions} \label{partitions-section}

A {\bf partition} is a weakly decreasing sequence of nonnegative integers $\lambda = (\lambda_{1}, \lambda_{2}, \ldots)$ such that $\lambda_k=0$ for all large enough $k$. Associated to a partition is its Ferrers diagram, which we draw on a fixed coordinate grid using ``Russian'' conventions, as shown in Figure~\ref{partition}. The {\bf dual partition} $\lambda'$ is obtained from $\lambda$ by reflecting the Ferrers diagram across the vertical line going through the origin. Equivalently, $\lambda'_j = \# \{ i \mid \lambda_i \ge j \}$. Let $\cP$ denote the set of all partitions, or equivalently the set of Ferrers diagrams.

The {\bf coordinates} of a box $b$ in a Ferrers diagram are the coordinates $(i,j)$ of the center of $b$, using the axes shown in Figure~\ref{partition}. Let $b = (i,j)$ be a box and $\lambda$ a partition. The {\bf arm length} of $b$ relative to $\lambda$ is $\arm_\lambda(b):= \lambda_i-j$. The {\bf leg length} of $b$ relative to $\lambda$ is $\leg_\lambda(b) = \lambda'_j - i$. These quantities are nonnegative if and only if $b \in \lambda$.

\begin{figure}

\setlength{\unitlength}{0.4cm}
\begin{center}
\begin{picture}(20,12)
\put(0,1){
\begin{picture}(20,11)

\put(10,0){\vector(1,1){10}}
\put(9,1){\line(1,1){8}}
\put(8,2){\line(1,1){7}}
\put(7,3){\line(1,1){7}}
\put(6,4){\line(1,1){5}}
\put(5,5){\line(1,1){5}}
\put(4,6){\line(1,1){2}}
\put(3,7){\line(1,1){1}}

\put(10,0){\vector(-1,1){10}}
\put(11,1){\line(-1,1){7}}
\put(12,2){\line(-1,1){6}}
\put(13,3){\line(-1,1){5}}
\put(14,4){\line(-1,1){5}}
\put(15,5){\line(-1,1){5}}
\put(16,6){\line(-1,1){3}}
\put(17,7){\line(-1,1){3}}
\put(18,8){\line(-1,1){1}}

\put(-1,10.5){$j$}
\put(21,10.5){$i$}

\put(9.6,-1){ $\tiny{{}_{0}}$}
\put(10,-1){ $\tiny{{}_{-1}}$}
\put(11,-1){ $\tiny{{}_{-2}}$}
\put(12,-1){ $\tiny{{}_{-3}}$}
\put(13,-1){ $\tiny{{}_{-4}}$}
\put(14,-1){ $\tiny{{}_{-5}}$}
\put(15,-1){ $\tiny{{}_{-6}}$}
\put(16,-1){ $\tiny{{}_{-7}}$}
\put(17,-1){ $\tiny{{}_{-8}}$}
\put(18,-1){ $\tiny{{}_{-9}}$}

\put(8.6,-1){ $\tiny{{}_{1}}$}
\put(7.6,-1){ $\tiny{{}_{2}}$}
\put(6.6,-1){ $\tiny{{}_{3}}$}
\put(5.6,-1){ $\tiny{{}_{4}}$}
\put(4.6,-1){ $\tiny{{}_{5}}$}
\put(3.6,-1){ $\tiny{{}_{6}}$}
\put(2.6,-1){ $\tiny{{}_{7}}$}
\put(1.6,-1){ $\tiny{{}_{8}}$}
\put(0.6,-1){ $\tiny{{}_{9}}$}

\thicklines

\put(10.7,3.8){$b$}

\put(-1.2,-1.15){$\cdots$}
\put(19.6,-1.15){$\cdots$}

\end{picture}}
\end{picture}
\end{center}
\vspace{0.3cm}

\caption{The Ferrers diagram of the partition $(7,6,5,5,5,3,3,1)$. The parts are the lengths of the ``rows'' of boxes sloping up and to the left. The $(x,y)$ coordinates are normalized so that the vertex of the partition has coordinates $(0.5, 0.5)$, and all boxes have unit side lengths. The center of each box has coordinate $(i,j)$ for some $i,j \in \bz$. For the box labeled $b$, $i=3$ and $j=2$. Here $\hook(b)= 8$, $\arm(b)=3$, and $\leg(b)=4$.
  \label{partition}}
\end{figure}
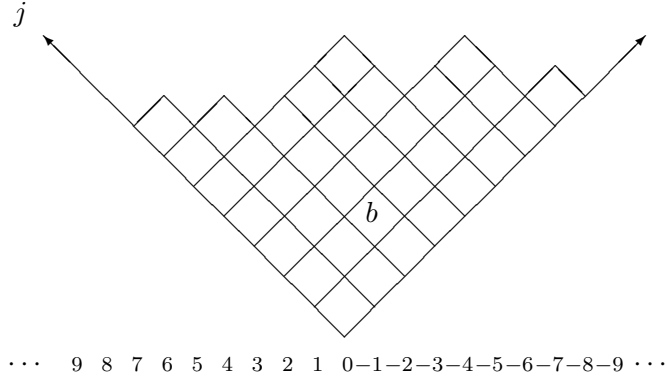

Let $\cP^\ell$ denote the set of $\ell$-tuples of partitions $\bl = (\lambda(1), \dots, \lambda(\ell))$, which we call multi-partitions. To distinguish boxes for different $\lambda(k)$ in a multi-partition $\bl$, we will sometimes use the notation $(k;i,j)$ to denote the box $(i,j)$ associated with $\lambda(k)$, so a box $b$ is associated with a triple of coordinates $(k_b; i_b, j_b)$.  A {\bf colored multi-partition} is a multi-partition $\bl = (\lambda(1), \ldots, \lambda(\ell))$ along with a chosen function $p \colon \{1, \dots, \ell \} \to \bz/n$. For $w= (w_{\overline 0}, \ldots, w_{\overline{n-1}})$, we say a colored multi-partition is of type $w$ if for all $\overline k \in \bz/n$, $\#\{ 1 \leq j \leq \ell \mid p(j) = \overline k \}=w_{\overline k}$.
  
\begin{Definition} \label{AR-def} 
Fix a colored multi-partition $\bl$. Given a box $b = (k;i,j) \in \lambda(k)$, the {\bf color} $\bar{c}(b)$ of $b$ is the residue of $p(k)-i+j$ modulo $n$.  Let $c(\bl) = (c_{\ol{0}}, \dots, c_{\ol{n} - \ol{1}})$ where $c_\ii$ is the number of $\ii$-colored boxes in $\bl$. Define $A(\bl)$ to be the set of boxes $b = (k;i,j)$ which can be added to $\lambda(k)$ so that the result is still a partition and $R(\bl)$ to be the set of boxes $b = (k;i,j)$ which can be removed from $\lambda(k)$ so that the result is still a partition. For each residue $\ii$ modulo $n$, define
\[
A_{\bar \imath}(\bl) = \{ b \in A(\bl) \mid \overline c(b)=\bar \imath \}, \quad R_{\bar \imath}(\bl)= \{ b \in R(\bl) \mid \overline c(b) = \bar \imath \},
\]
which we call the set of {\bf addable $\ii$-nodes} and {\bf removable $\ii$-nodes}, respectively.
\end{Definition}

\begin{Definition} \label{defn:slopedatum} 
For a fixed $\ell$, a {\bf slope datum} is an $(\ell+2)$-tuple of positive real numbers $\bx = (\xi_{\Omega}, \xi_{\overline{\Omega}}, \xi_1, \dots, \xi_\ell)$.
\end{Definition}

\begin{Definition} \label{defn:height}
Fix a multi-partition $\bl = (\lambda(1), \ldots, \lambda(\ell))$ and a slope datum $\bx$. We define the {\bf height} of a box $b = (k;i,j)$ by $h^\bx(b) = \xi_k + \xi_{\Omega} i + \xi_{\overline \Omega} j$. We call such a datum {\bf general} if $b \ne b'$ implies that $h^\bx(b) \ne h^\bx(b')$. We call such a datum {\bf integral} if $\xi_\Omega, \xi_{\overline \Omega}, \xi_1, \ldots, \xi_\ell \in \bz_{>0}$. We call such a datum {\bf aligned} if, for all $i,j$, $|\xi_i-\xi_j| < \xi_\Omega+ \xi_{\overline \Omega}$.
\end{Definition}

\begin{Definition} \label{def:multiillegal} A triple $(b, i, j)$ is called {\bf $\bx$-illegal} for $\bl$ if 
\begin{enumerate}
\item $b \in \lambda(i)$ 
\item $n$ divides $p(i) - p(j) + \arm_{\lambda(i)}(b) + \leg_{\lambda(j)}(b) +1$, and
\item \label{eqn:multiillegal} $-\xi_\Omega < \xi_j-\xi_i+ \xi_{\Omega} \leg_{\lambda(j)}(b) - \xi_{\ol{\Omega}} \arm_{\lambda(i)}(b) < \xi_{\overline \Omega}$.
\end{enumerate}
We say that $\bl$ is {\bf $\bx$-regular} if it contains no $\bx$-illegal triples.
\end{Definition}

\subsection{Torus actions} \label{ss:torus-combinatorics}
Fix $\ell > 0$, and $\Lambda = \Lambda_{\ii_1} + \cdots + \Lambda_{\ii_\ell}$, where $\Lambda_\ii$ is the $\ii$th fundamental weight. We now apply the framework from \S\ref{sec:framework} to realize $B(\Lambda),$ so fix a graded vector space with dimension $(w_\ii)_{\ii \in \bz/n}$, where $w_\ii = |\{ 1 \leq j \leq \ell \mid \ii_j = \ii \}|$. 

Let $T = \bc^* \times \bc^* \times T_W$, and consider the action of $T$ on $\mathfrak{M}(\bfv,W)$ given by 
\[
(t_\Omega, t_\oOmega, D) \cdot [x, s,t] = [x',s',t'],
\]
where $s'=t_\Omega t_\oOmega D^{-1}s$, $t'=t D$, $x'_a= t_\Omega x_a$ for $a \in \Omega$, and $x'_a= t_\oOmega x_a$ for $a \in \oOmega$. Fix an $I$-graded basis $\{w_1,\ldots, w_\ell\}$ for $W$ consisting of simultaneous eigenvectors for $T_W$, where $w_j \in W_{\ii_j}$. Then 
\[
\begin{aligned}
T & \cong (\bc^*)^{\ell+2} \\
(t_\Omega, t_\oOmega, D)& \leftrightarrow (t_\Omega, t_\oOmega, t_1, \ldots, t_\ell).
\end{aligned}
\]

\begin{Remark} \label{rem:sln:torus}
In the above construction, $T$ is subtorus of the torus $\cT$ from \S\ref{ssec:torus}. Thus the construction below is really the application of the methods from \S\ref{sec:framework} to the $\asl_n$ case. 
\end{Remark}

\begin{Remark} \label{rem:cons}
This same torus action has been considered in e.g. \cite{instantonI}. Here we have made a slight change of conventions; to match that paper we should instead define $s'=t_\Omega t_\oOmega Ds$, $t'=t D^{-1}$. However, our conventions are more convenient in \S\ref{s:asln-results} below. 
\end{Remark}

\begin{Proposition}[{\cite[Proposition 2.9]{instantonI}}] \label{prop:fixed-points} 
The fixed points of $\mathfrak{M}(\bfv,W)$ are indexed by $\ell$-multi-partitions $\bl$ of $|{\bf v}|$ with $c(\bl)= \bfv$, where the multi-partition ${\bl}$ is associated with the point $p_{\bl} \in \mathfrak{M}(\bfv,W)$ shown in Figure~\ref{fig:fixed}.
\end{Proposition}

The multi-partition for Proposition~\ref{prop:fixed-points} naturally comes with a coloring map $p$ from the component partitions to $\bz/n$, where the partition corresponding to $w_k$ is colored $\ii$ where $w_k \in W_\ii$. 

\begin{figure}
\begin{tikzpicture}[scale=0.6]

\draw 

(2,-1)--(-1,2)
(3,0)--(0,3)
(4,1)--(2,3)
(2,-1)--(4,1)
(1,0)--(3,2)
(0,1)--(2,3)
(-1,2)--(0,3)

(7,-1)--(9,1)
(6,0)--(8,2)
(5,1)--(6,2)
(7,-1)--(5,1)
(8,0)--(6,2)
(9,1)--(8,2)

(12,-1)--(10,1)
(13,0)--(11,2)
(14,1)--(12,3)
(12,-1)--(14,1)
(11,0)--(13,2)
(10,1)--(12,3)

(17,-1)--(15,1)
(18,0)--(16,2)
(17,-1)--(18,0)
(16,0)--(17,1)
(15,1)--(16,2)

(2,-2) node {$w_1$}
(7,-2) node {$w_2$}
(12,-2) node {$w_3$}
(17,-2) node {$w_4$}

(0,2) node {$\overline 2$}
(1,1) node {$\overline 1$}
(2,0) node {$\overline 0$}
(2,2) node {$\overline 0$}
(3,1) node {$\overline 2$}

(6,1) node {$\overline 2$}
(7,0) node {$\overline 1$}
(8,1) node {$\overline 0$}

(11,1) node {$\overline 2$}
(12,0) node {$\overline 1$}
(12,2) node {$\overline 1$}
(13,1) node {$\overline 0$}

(16,1) node {$\overline 0$}
(17,0) node {$\overline 2$}
;

\draw [->] (2.2,0.2)--(2.8,0.8);
\draw [->] (1.8,0.2)--(1.2,0.8);
\draw [->] (1.2,1.2)--(1.8,1.8);
\draw [->] (2.8,1.2)--(2.2,1.8);
\draw [->] (0.8,1.2)--(0.2,1.8);

\draw [->] (7.2,0.2)--(7.8,0.8);
\draw [->] (6.8,0.2)--(6.2,0.8);

\draw [->] (12.2,0.2)--(12.8,0.8);
\draw [->] (11.8,0.2)--(11.2,0.8);
\draw [->] (11.2,1.2)--(11.8,1.8);
\draw [->] (12.8,1.2)--(12.2,1.8);

\draw [->] (16.8,0.2)--(16.2,0.8);

\draw [->] (2,-1.7)--(2,-0.35);
\draw [->] (7,-1.7)--(7,-0.35);
\draw [->] (12,-1.7)--(12,-0.35);
\draw [->] (17,-1.7)--(17,-0.35);
\end{tikzpicture}

\caption{\label{fig:fixed} The fixed point $p_{\bm \lambda}$ associated to a multi-partition $\bl = ((3,2), (2,1), (2,2), (2))$. In this example, $n=3$, and we take $(w_{\overline 0}, w_{\overline 1}, w_{\overline 2})=(1,2,1)$, and $(v_{\overline 0}, v_{\overline 1}, v_{\overline 2})= (5,4,5).$ Each $\ii$ colored box of the multi-partition corresponds to a basis element of $V_\ii$. The symbols $w_k$ represent a basis for $W$ respecting the $I$-grading. When representing the $x_{\ii}$ and ${}_\ii x$ as matrices in this basis, the arrows give nonzero entries equal to $1$ (the color of the boxes at the head and tail of $a$ uniquely specify which arrow). The arrows pointing up represent matrix elements of $1$ for $t$. All other matrix elements are $0.$}

\end{figure}
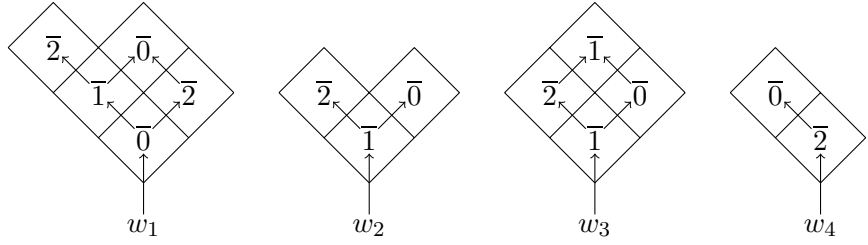

\subsection{Local coordinates} \label{ss:coordinates} 
Fix a $\bz/n$-graded vector space $W$ with $\dim W=\bfw$, and a colored multi-partition $\bl$ of type $\bfw$ with $c(\bl)=\bfv$. Introduce the type-specific notation  ${}_{\ii\pm \bar 1} x_\ii$  to mean $x_a$ where $a$ is the arrow from node $\ii$ to node $\ii\pm \bar 1$.

The torus fixed point $p_\bl$ from Figure \ref{fig:fixed} is a point in $\mL(\bfv, W)$. For any point $[V,x,s,t]$ in a neighborhood of $p_\bl$ in $\mM(\bfv, W)$,
\[
\{ b_{k;i,j}:=  {}_{p(k)-i+j} x_{p(k)-i+j-1} \circ \cdots \circ {}_{p(k)-i+1} x_{p(k)-i} \circ {}_{p(k)-i} x_{p(k)-i+1}  \circ \cdots \circ  {}_{{p(k)}-1} x_{p(k)}  \circ t(w_k) \mid (i,j) \in \lambda(k) \}
\]
forms a basis for $V$. This allows us to uniquely choose representatives $[\bar{V},x,s,t]$ for each point in this neighborhood: Let $\bar{V}$ be the vector space with basis indexed by $\{ (k;i,j): (i,j) \in \lambda(k)\}$. If $p$ has representative $[V,x,s,t]$, consider the vector space isomorphism $\gamma \colon V \rightarrow \bar{V}$ which takes $b_{k;i,j}$ to the basis element $(k;i,j)$. Then our chosen representative is $[\bar{V}, \gamma x \gamma^{-1}, s \gamma^{-1}, \gamma t ]$. Notice that this does not depend on the original choice of representative $[V, x,s,t]$. The only matrix elements in this chosen representative that differ from those of $p_\bl$ are:
\begin{enumerate}[leftmargin=20pt]
\item The matrix element of ${}_{\ii-\bar 1} \bar{x}_\ii$ from $b= b_{k;i,j}$ to $b'= b_{k';i',j'}$ where $\overline c (b') = \overline c(b)-1$, and either $j \neq 1$ or $b_{k; i+1, j} \not \in \bl$.

\item The matrix element of ${}_{\ii+ \bar 1} x_\ii$ from $b= b_{k;i,j}$ to $b'= b_{k';i',j'}$ where $\overline c (b') = \overline c(b)+1$ and $b_{k; i, j+1} \not \in \bl$.

\item The matrix coefficients of the maps $s_\ii$.  
\end{enumerate}

Let $\mE_\bl$ be the affine space with coordinates $\{e_{b \rightarrow b'}\}$ and $\{ e_{b'' \rightarrow r}\}$ such that  
\[
b=(k;i,j) \in \bl,\ b'=(k';i',j') \in \bl,\ b''=(k'';i'',j'')  \in \bl,\ p(r)= \overline c (b''),\;\; \text{ and either}
\]
\begin{itemize}[leftmargin=10pt]
\item $\overline c (b') = \overline c(b)+1$ and $(k; i, j+1) \not \in \bl$, or

\item $\overline c (b') = \overline c(b)-1$, and either $j \neq 1$ or $(k; i+1, j) \not \in \bl$.
\end{itemize}
We think of these coordinates as indexing the above matrix elements. Let $S_\bl$ be the reduced subvariety of $\mE_\bl$ consisting of those points which satisfy the preprojective relations. The previous paragraph shows that the natural map $m \colon S_\bl \rightarrow \mM(\bfv, W)$ is an injection. The image of $m$ consists of all $[V; x,s,t] \in \mM(\bfv, W)$ such that the elements $b_{k;i,j}$ are linearly independent, which is clearly an open condition, so in fact $m$ is a bijection from $S_\bl$ to an open subset of $\mM(\bfv, W)$ containing $p_\bl$. Furthermore, $m$ is birational \cite[Proposition 7.16]{harris}, and, since $\mM(\bfv, W)$ is smooth, $m$ is an isomorphism onto its image by Zariski's main theorem \cite[\S III.9, Original form]{mumford}. Then $m^{-1}$ is a local embedding of $\mM(\bfv,W)$ into $\mE_\bl$ near $p_\bl$. Via restriction, we get a local embedding of any irreducible component $Z$ of $\mL(\bfv, W)$ containing $p_\bl$ into $\mE_\bl$.

\begin{remark}
A similar description of local charts in the case $|\bfw|=1$ is given in \cite[Proposition 2.1]{haiman} using the language of Hilbert schemes, and the general case is discussed in \cite{Nakajima:1999}.
\end{remark}

Since we now have a chosen a canonical representative for each point in a neighborhood of each $p_\bl$, we can define the following useful subspace: Fix a slope datum $\bx$. For each $H \in \br$, 
\begin{equation} \label{VH}
V^{\geq H} = \text{span} \{ b_{k;i,j} : h^\bx(b) \geq H \} \quad \text{and} \quad V^{> H} = \text{span} \{ b_{k;i,j} : h^\bx(b) > H \}.
\end{equation}

\subsection{The $\asl_n$ quiver variety and the loop quiver} \label{sec:QLQ}

Consider the doubled quiver associated to the graph consisting of a single vertex and a single edge, which we denote by $\widetilde Q$:

\vspace{-.2in}
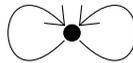
\begin{center}
\setlength{\unitlength}{0.3cm}
\begin{tikzpicture}

\draw (0,2)  node [shape=circle, fill=black, scale=0.6,draw] {} ;

\draw[-biggertip] (0.1,1.9) .. controls (1.1, 0.9) and ( 1.1, 3.1) .. (0.1,2.1);
\draw[-biggertip] (-0.1,1.9) .. controls (-1.1, 0.9) and ( -1.1, 3.1) .. (-0.1,2.1);

\end{tikzpicture}
\end{center}
\vspace{-.2in}
This does not satisfy the conditions of \S\ref{ssec:Nak-varieties}, but the definition of $\mathfrak{M}_{\widetilde Q}(v,W)$ still makes sense (although not all the statements of \S\ref{ssec:Nak-varieties} remain true); in fact, by \cite[Theorem 2.1]{Nakajima:1999}, the resulting variety can be identified with the so called ``punctual quot scheme." As explained in \cite[Chapter 4]{Nakajima:1999}, there is a geometric realization of the McKay correspondence that allows one to realize the quiver varieties for any symmetric affine Kac--Moody algebra using $\mathfrak{M}_{\widetilde Q}(v,W)$. We now explain this construction for the case of $\asl_n$. In \S\ref{sec-quot} below we will explain the meaning of our results in this language of quot schemes. 

For a point in $\mathfrak{M}_{\widetilde Q}(v,W)$, we use $x$ and $y$ to mean $x_a$ and $x_{\ol{a}}$ for the two arrows of the quiver. Fix an order $n$ linear automorphism $\psi$ of $W$ and a primitive $n^{\rm th}$ root of unity $\zeta$. There is an automorphism $\phi_n$ of $\mathfrak{M}_{\widetilde Q}(v,W)$ defined by
\begin{equation*}
x \mapsto \zeta x, \;\; y \mapsto \zeta^{-1} y, \;\; s \mapsto \psi \circ s, \;\; t \mapsto t \circ \psi^{-1}.  
\end{equation*}

For $i=0, \ldots, n-1$, let $W_i$ be the $\zeta^i$-eigenspace of $\psi$, and set $\bfw_i = \dim W_i$. For each $[V,x,s,t] \in \mathfrak{M}_{\widetilde Q}(v,W)^{\phi_n}$, there is an induced endomorphism on $V$: since $[V, x,s,t]$ and $[V, \zeta x, y \zeta^{-1}, \psi \circ s, t \circ \psi^{-1}]$ are equivalent under the $\GL(V)$ action, there is an element of $\GL(V)$ taking the first to the second. This endomorphism clearly has order $n$, so the eigenspace decomposition gives $V$ a $\bz/n$ grading. As in \cite[\S 4.2]{Nakajima:1999},
\begin{equation*}
\mathfrak{M}_{\widetilde Q}(v,W)^{\phi_n} \cong \coprod_{|{\bf v}| = v} \mathfrak{M}(\bfv, W),
\end{equation*}
where $\mathfrak{M}(\bfv, W)$ consists of those points where the dimension of each $\zeta^k$ eigenspace in $V$ is $v_k$.

The action of $T = (\bc^*)^2 \times (\bc^*)^\ell$ on $\mM(\bfv, W)$ is inherited from the action on $\mM_{\widetilde Q}(v, W)$ defined by
\[
(t_\Omega, t_\oOmega, D) \cdot [x,y, s,t] = [t_\Omega x, t_{\oOmega} y, t_\Omega t_\oOmega Ds, tD^{-1}].
\]
Furthermore, by \cite[Proposition 2.9]{instantonI}, the fixed points of the action on $\mM_{\widetilde Q}(v, W)$ are all in $\coprod_{|{\bf v}| = v} \mathfrak{M}(\bfv, W)$ and as such are indexed by multi-partitions.

Choose coordinates $(t_\Omega, t_{\overline \Omega}, t_1, \ldots, t_\ell)$ for $T$, where $(t_1, t_2, \ldots, t_\ell) = T_W$ and each $t_k$ is homogeneous with respect to the $\bz/n$-grading on $W$. Then $T$ acts on the tangent space to $p_\bl$ in  $\mM_{\widetilde Q}(v,W)$, and this action preserves the tangent space to $p_\bl$ in $\mathfrak{M}(\bfv,W)$. The next result follows from \cite[Theorem 2.11]{instantonI} (where we have modified the statement to match our conventions, as mentioned in Remark \ref{rem:cons}).

\begin{Proposition} \label{prop:char} 
The character of $T$ acting on the tangent space at $p_\bl$ in $\mathfrak{M}_{\widetilde Q}(v,W)$ is
\begin{equation*}
\sum_{\substack{(b, k, k') \\ b \in \lambda(k) }}
t_{k'}^{-1} t_k t_{\Omega}^{-\leg_{\lambda(k')}(b)} t_{\ol{\Omega}}^{\arm_{\lambda(k)}(b) + 1} 
+ t_{k'}t_k^{-1}  t_{\Omega}^{\leg_{\lambda(k')}(b) + 1} t_{\ol{\Omega}}^{-\arm_{\lambda(k)}(b)}. 
\end{equation*}
where the sum is over all triples $(b,k, k')$ for $k,k'=1,\dots,\ell$ and $b \in \lambda(k)$.
\end{Proposition}

\begin{Lemma} \label{lemma:tangentfixedpts} 
The tangent space to $\mathfrak{M}(\bfv,W) \subset \mathfrak{M}_{\widetilde Q}(v,W)$ is spanned by those basis vectors from Proposition \ref{prop:char} which satisfy the extra condition that
\begin{equation} \label{eq:is-l-hook}
p(k)-p(k') +\arm_{\lambda(k)}(b) + \leg_{\lambda(k')}(b)+1 \equiv 0 \pmod n.
\end{equation}
\end{Lemma}

\begin{proof}
Notice that the automorphism $\phi_n$ of $\mM_{\tilde Q}(v, W)$ is in fact the action of $(\zeta_n, \zeta_n^{-1}, D_n) \in T$, where $D_n$ acts on $W_i$ as $\zeta_n^{i}$. Thus the $\phi_n$ fixed subspace of the tangent space to $p_\bl$ consists exactly of those tangent vectors from Proposition \ref{prop:char} such that \eqref{eq:is-l-hook} holds. Thus we need only show that ${\rm T}_{p_\bl} (X^{\bz/n}) = ({\rm T}_{p_\bl} X)^{\bz/n}$, where $X = \mM_{\widetilde{Q}}(v,W)$. The inclusion $\subseteq$ is clear. For the other direction, we may find an analytic neighborhood around $p_\bl$ so that the action of $\bz/n$ is linear. Given $\gamma \in ({\rm T}_{p_{\bl}} X)^{\bz/n}$, choose a 1-parameter family $\gamma(t)$ whose derivative is $\gamma$. Then $n^{-1} \sum_{g \in \bz/n} g \cdot \gamma(t)$ lies in $X^{\bz/n}$ and its derivative is $\gamma$.
\end{proof}

\section{Resulting combinatorial realizations of $\asl_n$ crystal} \label{s:asln-results}

In this section we precisely describe some combinatorial models that arise from our construction in type $\asl_n$. Proofs are delayed until the next section. 

\subsection{The image of $M_\iota$} \label{sec:imageMiota}

Consider the $T$ action on $\mL(\bfv, W)$ from \S\ref{ss:torus-combinatorics}. Given an integral slope datum $\bx$, define a map 
\begin{equation} \label{eqn:zdefn}
\begin{split}
\iota_\bx \colon \bc^* & \rightarrow T \\ z & \mapsto (z^{\xi_{\Omega}}, z^{\xi_{\overline \Omega}}, z^{\xi_1}, \ldots, z^{\xi_\ell}),
\end{split}
\end{equation}
where we use the coordinates for $T$ from \S\ref{sec:QLQ}. We denote by $T_\bx$ the 1-torus $\bc^*$ along with the action on $\mM(\bfv, W)$ induced by the map $\iota_\bx$.

For a general slope datum $\bx$, one can choose a collection of integral slope data $(\bx^{(N)})_{N \in \bn}$ with the property that, for any multi-partition ${\bm \lambda}$ and any boxes $b,b' \in {\bm \lambda}$, $h^{\bx}(b) > h^{\bx}(b')$ if and only if, for all sufficiently large $N$, one has $h^{\bx^{(N)}}(b) > h^{\bx^{(N)}}(b')$. By looking at local coordinates in $\mathfrak{M}(\bfv, W)$ near each $p_\bl$ (see \S\ref{ss:coordinates}), it is clear that the family $\iota_{\bx^{(N)}}$ satisfies properties {\rm (F1)}, {\rm (F2)}, {\rm (F3)} from \S\ref{sec:framework}, and thus define an injective map 
\[
M_\bx \colon \coprod_\bfv \Irr \mathfrak{L}(\bfv,W) \rightarrow \cP^\ell.
\]
Furthermore, $M_\bx$ only depends on $\bx$, not the choice of the family $\bx^{(N)}$. As in \S\ref{ss:torus-combinatorics}, the multi-partitions are in fact colored by a function $p$ coming from the $\bz/n$-grading on $W$. 

\begin{Theorem} \label{thm:to-partitions}
Fix a general slope datum $\bx$. The image of $M_\bx$ is contained in the set of $\bx$-regular multi-partitions. If $\bx$ is aligned, then $\im M_\bx$ consists of exactly the $\bx$-regular multi-partitions of type prescribed by $p$.
\end{Theorem}

\subsection{Combinatorial description of the crystal structure on multi-partitions} \label{ss:comb-cryst}

For each integral highest weight $\Lambda$ and each general aligned slope datum $\bx$, Theorem \ref{thm:to-partitions} gives a bijection between $\coprod_\bfv \Irr \mathfrak{L}(\bfv,W)$ and the set of $\bx$-regular multi-partitions. Transporting the crystal structure from \S\ref{ss:quiver-crystal} gives a realization of the crystal $B(\Lambda)$ where the underlying set is the $\bx$-regular multi-partitions. We now give a purely combinatorial description of the crystal operators on $\bx$-regular multi-partitions.

Fix $\Lambda$ and a general slope datum $\bx$. For each multi-partition $\bl$, construct a string of brackets $S^\bx_{\bar \imath}(\bl)$ by placing a ``$($'' for every $b \in A_{\bar \imath}(\bl)$ and a ``$)$'' for every $b \in R_{\bar \imath}(\bl)$, in decreasing order of $h^\bx(b)$ from left to right. Cancel brackets by recursively applying the rule that adjacent matching parentheses $()$ get removed. Define operators $e^\bx_{\ii}, f^\bx_{\ii} \colon \cP^\ell \rightarrow \cP^\ell \cup \{ 0 \}$ . 
\begin{equation*} \label{Hcryst}
\begin{aligned}
e_{\ii}^\bx (\bl) &=
\begin{cases}
\lambda \setminus b & \text{if the first uncanceled ``$)$'' from the right in $S^\bx_{\bar \imath}(\bl)$ corresponds to $b$} \\
0 & \text{if there is no uncanceled ``$)$'' in $S^\bx_{\bar \imath}(\bl)$},
\end{cases} \\
f_{\ii}^\bx (\bl) &=
\begin{cases}
\lambda \cup b & \text{if the first uncanceled ``$($'' from the left in $S^\bx_{\bar \imath}(\bl)$ corresponds to $b$} \\
0 & \text{if there is no uncanceled ``$($'' in $S^\bx_{\bar \imath}(\bl)$}.
\end{cases}
\end{aligned}
\end{equation*}

\begin{Definition}
Let $B^\bx \subset \cP^\ell$ be the set of multi-partitions which can be obtained from the empty multi-partition $(\varnothing, \dots, \varnothing)$ by applying a sequence of operators $f_\ii^\bx$ for various $\ii$. 
\end{Definition}

\begin{Theorem} \label{thm:irrationalcrystal} 
Fix a general aligned slope datum $\bx$. With the notation above, $B^\bx = \im M_\bx$ (which by Theorem~\ref{thm:to-partitions} is the set of $\bx$-regular multi-partitions). The operators $e_\ii^\bx$ and $f_\ii^\bx$ are exactly the crystal operators inherited from the crystal structure on $\coprod_\bfv \Irr \mL(\bfv, W)$.
\end{Theorem}

\begin{Remark} 
One can easily see that the set of $\bx$-regular multi-partitions depends on $\bx$, so we do see many different combinatorial realizations of $B(\Lambda)$. In fact, one always obtains an uncountable family of realizations, since for any irrational number $1/(n-1) < z < n-1$ one can find a general  aligned slope datum $\bx$ with $\xi_\Omega/\xi_{\oOmega}=z$, and one can easily argue that these all lead to different realizations. 
\end{Remark}

\begin{Remark} \label{rem:Fayers1} 
When $\bl=\lambda$ is a single partition (i.e., in level $1$), the combinatorial operators defined above appeared in Fayers' recent work \cite{Fayers:2009}. To describe the exact relationship, set $y = \frac{n \xi_\Omega}{\xi_\Omega+ \xi_{\oOmega}}$ and consider the arm sequence $A_y$ from \cite[Lemma 7.4]{Fayers:2009}. Our notion of $\bx$-regular then agrees with Fayers' notion of $A_y$-regular, and our crystal operators agree with Fayers' crystal operators. 
\end{Remark}

\begin{Remark} \label{rem66}
The operators $e_\ii^\bx$ and $f_\ii^\bx$ are well-defined on any multi-partition, but they do not define the structure of an $\asl_n$-crystal on the set of all multi-partitions (this was noted in \cite[\S7]{Fayers:2009} in the level 1 case).
\end{Remark}

\subsection{Rational slopes}  \label{ss:rational}

There are other families $(\bx^{(N)})$ which satisfy (F1), (F2), (F3) from \S\ref{sec:framework}, and thus define an embedding $M_{\bx} \colon \coprod_\bfv \Irr \mathfrak{L}(\bfv,W) \rightarrow \cP^\ell$. For example, consider the following. Define two total orders on boxes as follows (set $b = (k_1;i_1,j_1)$ and $b' = (k_2;i_2,j_2)$):
\begin{enumerate}
\item $b \succ_\bx b'$ if and only if
\begin{itemize}
\item $h^{\bx}(b) > h^{\bx}(b')$ or
\item  $h^{\bx}(b) = h^{\bx}(b')$ and $i_1>i_2$ or 
\item  $h^{\bx}(b) = h^{\bx}(b')$ and $i_1=i_2$ and $k_1 > k_2,$
\end{itemize}
\item $b \succ'_\bx b'$ if and only if
\begin{itemize} 
\item $h^{\bx}(b) > h^{\bx}(b')$ or
\item $h^{\bx}(b) = h^{\bx}(b')$ and $i_1<i_2$ or 
\item $h^{\bx}(b) = h^{\bx}(b')$ and $i_1=i_2$ and $k_1 < k_2.$
\end{itemize}
\end{enumerate}
One can find families $\bx^{(N)}$ which satisfy either
\begin{enumerate}
\item $h^{\bx^{(N)}}(b) > h^{\bx^{(N)}}(b')$ for all sufficiently large $N$ if and only if $b \succ_\bx b'$ or
\item $h^{\bx^{(N)}}(b) > h^{\bx^{(N)}}(b')$ for all sufficiently large $N$ if and only if $b \succ'_\bx b'$.
\end{enumerate}
These lead to different maps $M \colon \coprod_\bfv \Irr \mathfrak{L}(\bfv,W) \rightarrow \cP^\ell$, and hence different combinatorial realizations of $B(\Lambda)$. When $\Lambda=\Lambda_0$ these two realization are the crystal structures corresponding to the two arm sequences associated to $y = \frac{n \xi_\Omega}{\xi_\Omega+ \xi_{\oOmega}}$ in \cite[Lemma 7.4]{Fayers:2009}. Since we will need it later on, we explicitly describe the combinatorial crystal structure for the order $\succ_\bx$.

\begin{Corollary} \label{cor:rat-crystal}
Define operators $\tilde e_i, \tilde f_i$ on $\cP^\ell$ as in \S\ref{ss:comb-cryst}, but ordering boxes according to $\succ_\bx$. Then the subset $B^\bx$ of $\cP^\ell$ which can be obtained from the empty multi-partition by applying operators $\tilde e_i$ and $\tilde f_i$ is a copy of $B(\Lambda)$.
\end{Corollary}

\begin{proof}
Follows immediately from Theorem \ref{thm:irrationalcrystal} by taking a limit. 
\end{proof}

\section{Proofs of results from Section \ref{s:asln-results}} \label{sec:proofs}

\subsection{Proof of Theorem~\ref{thm:to-partitions}} \label{sec:proof:to-partitions}

Fix a general slope datum $\bx$ and a sequence $\bx^{(k)}$ of integral slope data converging to $\bx$. Fix a colored multi-partition $\bl$ with $c(\bl)=\bfv$, and choose $N$ large enough so that for all pairs of boxes $b, b' \in \bl$, $h^{\bx^{(N)}}(b) > h^{\bx^{(N)}}(b')$ if and only if $h^{\bx}(b) > h^{\bx}(b')$. Fix a $\bz/n$-graded vector space $W$, where $\dim W_\ii= |\{ k : p(k) = \ii \}|$. 

\begin{Lemma} \label{lem:anl}
The dimension of the attracting set of $p_\bl$ under the action of $T_{\bx^{(N)}}$ on $\mM(\bfv, W)$ is $\dim \mM(\bfv, W)/2$ if and only if $\bl$ is $\bx$-regular. 
\end{Lemma}

\begin{proof}
Using the coordinates from \S\ref{ss:coordinates}, the fixed points of $T_{\bx^{(N)}}$ acting on $\mM(\bfv, W)$ are exactly the fixed points of $T$, and in particular they are isolated.

We will use the Bia\l ynicki-Birula decomposition \cite[\S 4]{BB} applied to $\mM(\bfv,W)$. This says that the set of points flowing to a given $T_{\bx^{(N)}}$-fixed point $p_\bl$ is an affine space whose dimension is the number of negative eigenvalues of $T_{\bx^{(N)}}$ acting on the tangent space of $p_\bl$. By Lemma~\ref{lemma:tangentfixedpts}, the weights of $T_{\bx^{(N)}}$ acting on ${\rm T}_{p_\bl} \mM(\bfv,W)$ come in pairs that are indexed by triples $(b, k, k')$ with $b \in \lambda(k)$ such that $p(k) - p(k') + \arm_{\lambda(k)}(b)+\leg_{\lambda(k')}(b) + 1 \equiv 0 \pmod{n}$. The two eigenvalues given by such a triple are
\begin{align*}
-\xi_{k'}+\xi_k - \xi_\Omega \leg_{\lambda(k')}(b) + \xi_\oOmega (\arm_{\lambda(k)}(b) + 1), \\
\xi_{k'} -\xi_k + \xi_\Omega (\leg_{\lambda(k')}(b) + 1) - \xi_\oOmega \arm_{\lambda(k)}(b).
\end{align*}
The sum of these two eigenvalues is $\xi_\Omega+\xi_{\overline \Omega} > 0$, so at least one is positive, and both of them are positive if and only if $(b, k, k')$ is $\bx$-illegal. 
\end{proof}
  
\begin{Lemma} \label{lem:in-L}
If $\bx$ is aligned, then the attracting set of $p_\bl$ under the action of $T_{\bx^{(N)}}$ is contained in $\mL(\bfv, W)$.
\end{Lemma}

\begin{proof}
Let $[x,s,t]$ belong to the attracting set of $p_\bl$. Since $\xi_\Omega, \xi_{\overline \Omega}$ are both positive $x$ must be nilpotent, since otherwise some path acts with a non-zero eigenvalue and this eigenvalue goes to infinity as $z \to \infty$ ($z$ as in \eqref{eqn:zdefn}). Furthermore, we must have $s=0$, since otherwise it follows from the stability condition that there is some path $\pi$ in $Q$ such that $s \circ \pi \circ t \neq 0$, and using the alignment condition this map goes to infinity as $z \to \infty$. Thus $[x,s,t] \in \mL(\bfv, W)$ by definition.
\end{proof}

Since $\mL(\bfv, W)$ is equidimensional of dimension $\dim \mM(\bfv, W)/2$, Lemma \ref{lem:anl} implies that $\im M_\bx$ is contained in the set of $\bx$-regular multi-partitions. Lemma \ref{lem:anl} also shows that the attracting set of $p_\bl$ for any $\bx$-regular $\bl$ has dimension exactly $\dim \mM(\bfv, W)/2$, so, in the case when $\bx$ is aligned, Lemma \ref{lem:in-L} shows that $\im M_\bx$ consists exactly of all $\bx$-regular multi-partitions. This completes the proof of Theorem~\ref{thm:to-partitions}.


\subsection{Proof of Theorem \ref{thm:irrationalcrystal}}

Fix $Z \in\Irr \mL(\bfv, W)$. Fix a general aligned slope datum $\bx$, and let $\bl = M_\bx(Z)$. Fix $\ii \in \bz/n$. We must show that
\begin{equation} \label{eq:2coa}
M_\bx(e_\ii(Z)) = e_\ii^\bx \bl.
\end{equation}
To do this, we construct a tangent vector to $Z$ at $p_\bl$ corresponding to each canceling pair of brackets in $S_\ii^\bx(M_\bx(Z))$ which ``witnesses'' the canceling of brackets (see Proposition \ref{prop:flatfamily}). The geometric definition of the crystal operators in \S\ref{ss:quiver-crystal} implies that $M_\bx(e_\ii(Z))$ differs from $M_\bx(Z)$ by removing the highest box $b$ such that $\dim \ker x_\ii|_{V_\ii^{\geq h^\bx(b)}} = \dim \ker x_\ii$ (see \eqref{VH} for this notation). By moving a small amount in the direction of each of the tangent vectors discussed above, we find an element where this box corresponds to the first uncanceled ``$)$'' in $S_\ii^\bx(M_\bx(Z))$ from the right. We then use a semi-continuity argument to show that this happens generically. 

First, we need several technical lemmas: 

\begin{Lemma} \label{lem:ar-illegal} 
Fix a multi-partition $\bl$. If there exists a pair $(r,a)$ where $r$ is a removable $\ii$-node and $a$ is an addable $\ii$-node with
\begin{align*}
h^\bx(r) + \xi_{\Omega}+ \xi_{\overline \Omega} > h^\bx(a) > h^\bx(r),
\end{align*}
then $\bl$ has a $\bx$-illegal triple.
\end{Lemma}

\begin{Remark} 
Lemma~\ref{lem:ar-illegal} can be interpreted geometrically: If $\bl$ is $\bx$-regular then, for any $\ii \in \bz/n$, $a \in A_\ii(\bl)$ and $r \in R_\ii(\bl)$, comparing the height function  $h^\bx$ on the centers of $a$ and $r$ orders them in the same way as comparing $h^\bx$ on the top corner of $r$ with $h^\bx$ on the bottom corner of $a$. 
\end{Remark}

\begin{proof}[Proof of Lemma \ref{lem:ar-illegal}]
Fix such a pair $(r,a)$, which we write in coordinates as $r = (k; i,j)$ and $a = (k'; i',j')$. Let $b_1= (i,j')$ and $b_2 = (i',j)$. If $j<j'$, then $b_2 \in \lambda(k')$. We claim that $(b_2, k', k)$ is $\bx$-illegal. To see this, first notice that 
\[
p(k') - p(k) + \arm_{\lambda(k')}(b_2) + \leg_{\lambda(k)}(b_2) + 1 = (p(k') - i' + j') - (p(k) - i + j),
\]
which is $0$ modulo $n$ since both $a$ and $r$ are $\ii$-nodes. Next, note that
\begin{equation*} 
\xi_k - \xi_{k'} + \xi_\Omega \leg_{\lambda(k)}(b_2) - \xi_{\ol{\Omega}} \arm_{\lambda(k')}(b_2) = h^\bx(r) - h^\bx(a) +\xi_{\oOmega}.
\end{equation*}
thus by Definition~\ref{def:multiillegal}\eqref{eqn:multiillegal}, $(b_2, k', k)$ is $\bx$-illegal if and only if 
\begin{equation} \label{4.5-2} 
-\xi_\Omega < h^\bx(r) - h^\bx(a) +\xi_{\oOmega} < \xi_\oOmega. 
\end{equation}
We are given that $h^\bx(r) + \xi_{\ol{\Omega}} + \xi_{\Omega} > h^\bx(a)$. Subtracting $h^\bx(a) +\xi_\Omega$ from both sides gives the left inequality in \eqref{4.5-2}. We are also given that $h^\bx(a) > h^\bx(r)$. Adding $\xi_{\oOmega} - h^\bx(a)$ to both sides gives the right inequality in \eqref{4.5-2}. Hence $(b_2, k', k)$ is in fact illegal.

If $j \geq j'$, then $b_1 \in \lambda(k)$ and a similar argument shows that $(b_1, k, k')$ is $\bx$-illegal. 
\end{proof}

Recall from \eqref{VH} that $V_\ii^{\geq H}$ is the span of the set of basis vectors of $V$ at height $\geq H$, using the coordinates from \S\ref{ss:coordinates}, so $\dim V_\ii^{\geq H}$ is  the number of $\ii$-boxes in $\bl$ of height at least $H$. Define $R_{\ii}^{\ge H}$ be the number of removable $\ii$-nodes at height at least $H$, and define $A_{\ii}^{\ge H}$ to be the number of addable $\ii$-nodes at height at least $H$. 

\begin{Lemma} \label{lem:illegal-squares} 
Fix a general slope datum $\bx$ and fix $H > \max_k \{ \xi_k \}$. Then
\begin{equation} \label{eq:add-a-box} 
\dim V_\ii^{\geq H} + \dim V_\ii^{\geq H + \xi_\Omega+ \xi_{\overline \Omega}} - \dim V_{\ii+\bar 1}^{\geq H+\xi_\Omega} - \dim V_{\ii-\bar 1}^{\geq H+\xi_{\overline \Omega}}
= R_\ii^{\geq H} - A_\ii^{\geq H+\xi_{\Omega} + \xi_{\overline \Omega}}.
\end{equation}
\end{Lemma}

\begin{proof}
The formula is true if $\bl$ has no boxes of height $\geq H$, as all terms are $0$. So truncate at this height, and then start adding the boxes back in order of height (so in particular after each box is added one always sees a valid  multi-partition). One can check that each box you add has the same effect on each side of \eqref{eq:add-a-box}:
\begin{itemize}[leftmargin=10pt]
\item If an $(\ii+\bar 1)$-box is added at height less than $H+\xi_\Omega$ or an $(\ii-\bar 1)$-box is added at height less than $H+\xi_\oOmega$, there is no change. 
  
\item If an $(\ii+\bar 1)$-box is added at height at least $H+\xi_\Omega$ or an $(\ii-\bar 1)$-box is added at height at least $H+\xi_\oOmega$, then both sides decrease by $1$. 

\item If an $\ii$-box is added at height less than $H + \xi_{\Omega} + \xi_{\overline \Omega}$, then both sides increase by $1$.
  
\item If an $\ii$-box is added at height at least $H + \xi_{\Omega} + \xi_{\overline \Omega}$, then both sides increase by $2$.
\end{itemize}
 
\noindent Thus once $\bl$ has been completely reconstructed \eqref{eq:add-a-box} still holds. 
\end{proof}

Now fix a family $(\iota_{\bx^{(N)}})_{N \geq 0}$ satisfying (F1)-(F3), and choose $N$ large enough so that $F_{\iota_{\bx^{(N)}}}(\bfv, W)=\mF(\bfv, W)$, and $M_{\bx^{(N)}} = M_\bx$. The action of $T_{\bx^{(N)}}$ on $\mM(\bfv, W)$ naturally extends to an action on the affine space $\mE_{\bl}$ from \S\ref{ss:coordinates}, where the action is described on coordinates by 
\begin{itemize}[leftmargin=10pt]
\item If $c(b')=c(b)+ \overline{1}$, then $z \cdot e_{b\rightarrow b'}= z^{  \xi^{(N)}_k +\xi^{(N)}_\Omega i + \xi^{(N)}_{\oOmega} j -\xi^{(N)}_{k'} - \xi^{(N)}_\Omega i' - \xi^{(N)}_{\oOmega} j'+ \xi^{(N)}_\oOmega } e_{b\rightarrow b'}$, 

\item  If $c(b')=c(b)- \overline{1}$, then $z \cdot e_{b\rightarrow b'}= z^{ \xi^{(N)}_k +\xi^{(N)}_\Omega i + \xi^{(N)}_{\oOmega} j -\xi^{(N)}_{k'} - \xi^{(N)}_\Omega i' - \xi^{(N)}_{\oOmega} j'+ \xi^{(N)}_\Omega } e_{b\rightarrow b'}$,

\item $z \cdot e_{ b \rightarrow r} =z^{ -\xi_k -\xi^{(N)}_\Omega i - \xi^{(N) }_\oOmega j +\xi_{p(r)}} e_{b \rightarrow r}.$
\end{itemize}
Clearly the attracting set of $p_\bl$ in $\mE_\bl$ consists of points such that, 
\begin{align}
& \text{If $h^\bx(b') \leq h^\bx(b)+ \xi_\oOmega$, then $e_{b \rightarrow b'}=0$.}\\
& \text{If $h^\bx(b') \leq h^\bx(b)+ \xi_\Omega$,  then $e_{b \rightarrow b'}=0$.}  \\
& \text{If $h^\bx(b) \leq \xi_{p(r)}$,  then $e_{b \rightarrow r}=0$.}  
\end{align}
By Lemma~\ref{lem:in-L} the attracting set for $p_\bl$ in $\mathfrak{M}(\bfv, W)$ is contained in $\mathfrak{L}(\bfv, W)$, and hence is exactly $Z$. It follows that, in a neighborhood of $p_\bl$, $Z$ is defined by these equations along with the defining equations of $\mM(\bfv, W)$ inside $\mE_\bl$. In particular, for all $[V, x, 0,t] \in Z$,
\begin{equation} \label{eq:im}
\im {}_\ii x|_{V_{\ii+\bar 1}^{\geq H+\xi_\Omega} \oplus  V_{\ii-\bar 1}^{\geq H+\xi_{\overline \Omega}}} \subseteq V_\ii^{\geq H+\xi_\Omega+ \xi_{\overline \Omega}}.
\end{equation}

\begin{Lemma} \label{lem:new} 
For any $[V, x, 0, t] \in Z$  and $H > \max_k \{ \xi_k \}$, $\dim \ker x_\ii|_{V_\ii^{\geq H}} \geq R_\ii^{\geq H} -A_\ii^{\geq H +\xi_{\Omega} + \xi_{\overline \Omega}}$.
\end{Lemma}

\begin{proof}
For the special point $p_\bl$, \eqref{eq:im} holds with equality, so by semi-continuity it holds with equality on an open dense subset.  Thus generically
\[
\dim \ker {}_\ii x|_{V_{\ii+\bar 1}^{\geq H+\xi_\Omega} \oplus  V_{\ii-\bar 1}^{\geq H+\xi_{\overline \Omega}}}  = \dim V_{\ii+\bar 1}^{\geq H+\xi_\Omega} + \dim  V_{\ii-\bar 1}^{\geq H+\xi_{\overline \Omega}} - \dim V_\ii^{\geq  H+\xi_\Omega+ \xi_{\overline \Omega}}. 
\]
As in \S\ref{ss:quiver-crystal}, the image of 
\[
  x_\ii \colon V_\ii^{\geq H} \rightarrow V_{\ii+\bar 1}^{\geq H+\xi_\Omega} \oplus  V_{\ii-\bar 1}^{\geq H+\xi_{\overline \Omega}}
\]
is contained in the kernel of ${}_\ii x$. Hence generically
\begin{align*}
\dim \ker x_\ii|_{  V_\ii^{\geq H}}  & \geq \dim  V_\ii^{\geq H} - \dim \ker {}_\ii x|_{V_{\ii+\bar 1}^{\geq H+\xi_\Omega} \oplus V_{\ii-\bar 1}^{\geq H+\xi_{\overline \Omega}}} \\ 
& = \dim V_\ii^{\geq H} + \dim V_\ii^{\geq H + \xi_\Omega+ \xi_{\overline \Omega}} - \dim V_{\ii+\bar 1}^{\geq H+\xi_\Omega} - \dim V_{\ii-\bar 1}^{\geq H+\xi_{\overline \Omega}}\\
& =R_\ii^{\geq H} -A_\ii^{\geq H +\xi_{\Omega} + \xi_{\overline \Omega}},
\end{align*}
where the last equality is by Lemma~\ref{lem:illegal-squares}. So we know that this inequality holds on a dense open subset. A lower bound on the dimension of the kernel can be rephrased as the vanishing of minors of a certain size, and so the inequality must hold everywhere. 
\end{proof}

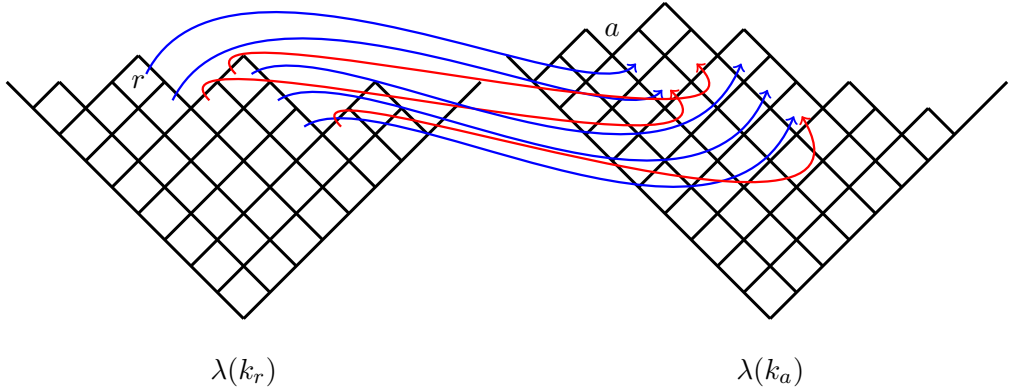
\begin{figure}
\begin{tikzpicture}[scale=0.35]

\draw[line width = 0.04cm] (0,0)--(9,9);
\draw[line width = 0.04cm] (-1,1) -- (6,8);
\draw[line width = 0.04cm] (-2,2) -- (5,9);
\draw[line width = 0.04cm] (-3,3) -- (2,8);
\draw[line width = 0.04cm] (-4,4) -- (1,9);
\draw[line width = 0.04cm] (-5,5) -- (0,10);
\draw[line width = 0.04cm] (-6,6) -- (-3,9);
\draw[line width = 0.04cm] (-7,7) -- (-4,10);
\draw[line width = 0.04cm] (-8,8) -- (-7,9);

\draw[line width = 0.04cm] (0,0) -- (-9,9);
\draw[line width = 0.04cm] (1,1) -- (-7,9);
\draw[line width = 0.04cm] (2,2) -- (-5,9);
\draw[line width = 0.04cm] (3,3) -- (-4,10);
\draw[line width = 0.04cm] (4,4) -- (-1,9);
\draw[line width = 0.04cm] (5,5) -- (0,10);
\draw[line width = 0.04cm] (6,6) -- (4,8);
\draw[line width = 0.04cm] (7,7) -- (5,9);

\draw[line width = 0.04cm] (20,0) -- (29,9);
\draw[line width = 0.04cm] (19,1) -- (26,8);
\draw[line width = 0.04cm] (18,2) -- (24,8);
\draw[line width = 0.04cm] (17,3) -- (23,9);
\draw[line width = 0.04cm] (16,4) -- (21,9);
\draw[line width = 0.04cm] (15,5) -- (20,10);
\draw[line width = 0.04cm] (14,6) -- (19,11);
\draw[line width = 0.04cm] (13,7) -- (17,11);
\draw[line width = 0.04cm] (12,8) -- (16,12);
\draw[line width = 0.04cm] (11,9) -- (13,11);

\draw[line width = 0.04cm] (20,0) -- (10,10);
\draw[line width = 0.04cm] (21,1) -- (12,10);
\draw[line width = 0.04cm] (22,2) -- (13,11);
\draw[line width = 0.04cm] (23,3) -- (15,11);
\draw[line width = 0.04cm] (24,4) -- (16,12);
\draw[line width = 0.04cm] (25,5) -- (19,11);
\draw[line width = 0.04cm] (26,6) -- (23,9);
\draw[line width = 0.04cm] (27,7) -- (26,8);

\draw (-4, 9) node {$r$};
\draw (14, 11) node {$a$};

\draw[line width = 0.03cm, ->, color=blue] (-3.7, 9.3) .. controls (-0.7, 15.3) and (13.2,7.7) .. (14.9,9.7);
\draw[line width = 0.03cm, ->, color=blue] (-2.7, 8.3) .. controls (0.3, 14.3) and (14.2,6.7) .. (15.9,8.7);

\draw[line width = 0.03cm, ->, color=red] (-1.3, 8.3) .. controls (-4.3, 11.3) and (20.2,4.7) .. (16.2,8.7);
\draw[line width = 0.03cm, ->, color=red] (-0.3, 9.3) .. controls (-3.3, 12.3) and (21.2,5.7) .. (17.2,9.7);

\draw[line width = 0.03cm, ->, color=blue] (0.3, 9.3) .. controls (3.3, 11.3) and (16.2,3) .. (18.9,9.7);
\draw[line width = 0.03cm, ->, color=blue] (1.3, 8.3) .. controls (4.3, 10.3) and (17.2,2) .. (19.9,8.7);
\draw[line width = 0.03cm, ->, color=blue] (2.3, 7.3) .. controls (5.3, 9.3) and (18.2,1) .. (20.9,7.7);

\draw[line width = 0.03cm, ->, color=red] (3.7, 7.3) .. controls (0.7, 10.3) and (25.2,1) .. (21.2,7.7);

\draw(0,-2) node {$\lambda(k_r)$};

\draw(20,-2) node {$\lambda(k_a)$};

\end{tikzpicture}

\caption{\label{fig:tangent-vectors} 
The one-parameter family in $\mE_\bl$ from Proposition \ref{prop:flatfamily}. The blue arrows are matrix elements for ${}_{\overline{k}-1} x_{\overline{k}}$, and the red arrows are matrix elements for  ${}_{\overline{k}+1} x_{\overline{k}}$, for various $\overline{k}$. All other coordinates are $0$. In general, to define these new matrix elements, one follows the rim of $\lambda(k_r)$, beginning at $r$, putting in matrix elements of $-\eps$ as shown. As soon as one would need to draw an arrow to a box which is not present in $\lambda(k_a)$ one may stop, or if one reaches the end of the rim of $\lambda(k_r)$ one may stop. One of these must happen at some point because $j_a-1 \geq j_r$. 
}
\end{figure}

We are now ready to construct the tangent vectors we need:

\begin{Proposition} \label{prop:flatfamily} 
Fix $Z \in \Irr \mL(\bfv, W)$ and let $\bl = M_\bx(Z)$. Pick $\alpha = (a,r)$ where $a \in A_\ii(\bl)$ and $r \in R_\ii(\bl)$ are such that $h^\bx(a) > h^\bx(r)$.  Let $a'$ be the box with coordinates $(k_a;i_a-1,j_a)$ and $a''$ the box with coordinates $(k_a;i_a,j_a-1)$. Then there is an embedding $d_\alpha \colon \bc \hookrightarrow \mE_\bl$ such that $d_\alpha(0)= p_\bl$, and, for all $\eps \in \bc^*$, the coordinates for $d_\alpha(\eps)$ as in \S\ref{ss:coordinates} satisfy 
\begin{enumerate}[label={\rm (\roman*)}, leftmargin=20pt]
\item \label{lb1} $x^{r, a'} -x^{r,a''} = \eps$.

\item \label{lb2} For all other pairs $\tilde a \in A_\ii(\bl)$ and $\tilde r \in R_\ii(\bl)$, $x^{\tilde r, \tilde a'} -x^{\tilde r,\tilde a''} = 0$, where $\tilde a'$ is the box with coordinates $(k_{\tilde a};i_{\tilde a}-1,j_{\tilde a})$ and $\tilde{a}''$ is the box with coordinates $(k_{\tilde a};i_{\tilde a},j_{\tilde a}-1)$.

\item \label{pt} The resulting tangent vector $t_{a,r}$ in ${\rm T}_{p_\bl} \mE_\bl$ lies in ${\rm T}_{p_\bl} Z$. 
\end{enumerate}
Here, by convention, we set $x^{b,b'} = 0$ if any of the coordinates of $b'$ are negative.
\end{Proposition}

\begin{proof}
Since $\bl$ is $\bx$-regular, the condition $h^\bx(a) > h^\bx(r)$ together with Lemma \ref{lem:ar-illegal} implies that $h^\bx(a) \ge h^\bx(r)+ \xi_\Omega+\xi_\oOmega$. This can be written as
\[
\xi_\Omega(i_a - i_r) + \xi_{\ol{\Omega}}(j_a - j_r) \ge \xi_{k_r} - \xi_{k_a} + \xi_\Omega + \xi_{\ol{\Omega}}.
\]
If both $i_a \le i_r$ and $j_a \le j_r$, then both sides of the inequality are non-positive, which implies that $|\xi_{k_a} - \xi_{k_r}| \ge \xi_{\Omega} + \xi_{\ol{\Omega}}$. But since $\bx$ is assumed to be aligned, this does not happen. So we know that either $i_a-1 \geq i_r$ or $j_a-1 \geq j_r$. 
  
In the second case, consider the family in $\mE_\bl$ given in Figure \ref{fig:tangent-vectors}. The defining equations of $\mM(\bfv, W)$ inside $\mE_{\bl}$ (i.e., the preprojective relations) are satisfied modulo $\eps^2$, so this family defines a tangent vector to $\mM(\bfv, W)$. Furthermore, this tangent vector is in the attracting set of $p_\bl$ in $\mE_\bl$, so it is in fact a tangent vector to the irreducible component $Z$.
  
The first case is handled by a symmetric construction.
\end{proof}
  
Following the definition of the crystal operators in \S\ref{ss:quiver-crystal}, to obtain a generic point in $e_\ii(Z)$, one should proceed as follows: start with a generic point $z \in Z$, and choose a representative $[V,x,t]$.  Take the  quotient of this representation by a generic 1-dimensional module in the $\ii$-socle of $V$ (that is, the socle of $z$ intersected with $V_\ii$) to obtain $[V', x',t']$. This will be a representative of a generic point $z'$ in $e_\ii(Z)$. 

The $\ii$-socle of $V_\ii$ is exactly $\ker x_\ii$. Quotienting out by a element of $\ker x_\ii$ will decrease $\dim V_\ii^{\geq H}$ by $1$ for low heights $H$, and will not change $\dim V_\ii^{\geq H}$ for high values of $H$. If this is done generically, the cutoff value of $H$ between these two behaviors will be as low as possible, which is to say that $\dim V_\ii^{\geq H}$ will be unchanged for all $H>\bar{H}$, where $\bar{H}$ is the lowest value of $H$ such that $\dim \ker x_\ii^{>H} < \dim \ker x_\ii^{\geq H}$. 

By the definition of the torus action, there is an $\ii$ box at height $H$ in $\lim_{z \rightarrow \infty} z \cdot (V,x,t)$ exactly when 
\begin{equation} \label{eq:H-diff}
\dim V_\ii^{>H} < \dim \ker V_\ii^{\geq H},
\end{equation} 
and similarly for $(V',x',t')$. The only height where \eqref{eq:H-diff} differs for $V$ and $V'$ is $\bar H$ from the previous paragraph. Thus $z'$ will flow to $p_{\bl'}$, where $\bl'$ is obtained from $\bl$ by removing the box at height $\bar H$. That is, $\bl'$ is obtained from $\bl$ by removing the lowest $\ii$ colored box $b$ such that $\dim \ker x_\ii|_{V_\ii^{> h^{\bx}(b)}} < \dim \ker x_\ii$. Thus to establish \eqref{eq:2coa}, it suffices to show that
\begin{enumerate}[leftmargin=20pt]
\item For generic $[x,0,t] \in Z$, $\dim \ker x_\ii$ is the number of uncanceled ``)'' brackets in $S_\ii^\xi(\bl)$ (see \S\ref{ss:comb-cryst}). 
  
\item Let $r$ be the removable box corresponding to the first uncanceled ``)'' from the right in $S_\ii^\xi(\bl)$, and let $H = h^\bx(r)$. Then $\dim \ker x_\ii|_{V_\ii^{\geq H}} = \dim \ker x_\ii$ and $\dim \ker x_\ii|_{V_\ii^{> H}}=\dim \ker x_\ii-1$. \end{enumerate}
  
Theorem~\ref{thm:irrationalcrystal} includes the hypothesis that $\bx$ is aligned which implies $H > \max_i \{\xi_i\}$, so, by Lemma~\ref{lem:new}, $\dim \ker x_\ii|_{V_\ii^{\geq H}} \geq $ $R_\ii^{\geq H} -A_\ii^{\geq H +\xi_{\Omega} + \xi_{\overline \Omega}}$. By Theorem~\ref{thm:to-partitions} the multi-partition $\bl$ is $\bx$-regular, so by Lemma~\ref{lem:ar-illegal} we see that $A_{\ii}^{\ge H+\xi_\Omega + \xi_{\ol{\Omega}}} = A_{\ii}^{\ge H}$. This implies that 
\begin{equation} \label{eq:bound1}
\dim \ker x_\ii|_{V_\ii^{\geq H}} \ge \varphi \quad  \text{and} \quad  \dim \ker x_\ii|_{V_\ii^{> H}} \ge \varphi - 1,
\end{equation} 
where $\varphi$ is the number of uncanceled ``)''. In particular, we also have
\begin{equation} \label{eq:bound2}
\dim \ker x_{\ii} \ge \varphi.
\end{equation}
By semi-continuity, to prove (ii), it is enough to find one element of $Z$ where all the inequalities in \eqref{eq:bound1} and \eqref{eq:bound2} hold with equality. For each canceling pair $(a_j, r_j)$, let $t_{a_j, r_j}$ be the tangent vector to $p_\bl$ in $Z$ defined by the embedding of $\bc^*$ from Proposition \ref{prop:flatfamily}\eqref{pt}. Moving a small amount in the direction of a generic linear combination $\sum c_j t_{a_j, r_j}$ gives the desired element, where the sum is over all canceling pairs of brackets and $t_{a_j, r_j}$ is the tangent vector from Proposition \ref{prop:flatfamily}. This completes the proof of Theorem \ref{thm:irrationalcrystal}.

\section{Application to the monomial crystal} \label{sec:monomial}

\subsection{Background on the monomial crystal for $\asl_n$}   \label{ss:monom}

Here we describe the monomial crystal of Nakajima \cite[\S3]{Nakajima:2003} in the case of $\asl_n$, as generalized and modified by Kashiwara in \cite[\S4]{Kashiwara:2001}. In fact the definition below is even more general than Kashiwara's, since Kashiwara assumes that our $K$ is $1$. 

We work with the set $\mathcal M$ of Laurent monomials in variables $Y_{\ii, k}$ for $\ii \in \bz/n$, $k \in \bz$. That is, an element of $\mathcal M$ is a product $M= \prod_{\ii \in \bz/n,\ k \in \bz} Y_{\ii,k}^{y_{\ii,k}}$ where each $y_{\ii,k}$ is an integer, and all but finitely many are $0$. For each $\ii \in \bz/n$, fix integers $c_{\ii, \ii+1}$ and $c_{\ii+1, \ii}$ such that the quantity $K= c_{\ii,\ii+1}+ c_{\ii+1,\ii}$ is constant over $\ii \in \bz/n$. We denote this data by ${\bf c}$. Define
\[
A_{{\bf c}; \ii,k}= Y_{\ii,k} Y_{\ii, k+K}  Y_{\ii+1, k+c_{\ii,\ii+1}}^{-1} Y_{\ii-1, k+c_{\ii,\ii-1}}^{-1}.
\]
For a monomial $\displaystyle M= \prod_{\ii \in \bz/n,\ k \in \bz} Y_{\ii,k}^{y_{\ii,k}}$, let
\begin{align*}
\wt(M) &= \sum_{\ii,k} y_{\ii,k} \Lambda_\ii, \qquad
\varepsilon_\ii(M) = -  \min_{k \in \bz} \big(   \sum_{s > k} y_{\ii,s} \big), \qquad
\varphi_\ii(M) = \max_{k \in \bz} \big( \sum_{s \leq k} y_{\ii,s} \big),\\
k_e(M) &= \max \big\{ k : \varepsilon_\ii(M)= - \sum_{s > k} y_{\ii,s} \big\}, \qquad
k_f(M) = \min \big\{ k : \varphi_\ii(M)= \sum_{s \leq k} y_{\ii,s} \big\}.
\end{align*}
Define operators $\tilde e_\ii^{\bf c}, \tilde  f_\ii^{\bf c} \colon \mathcal M \rightarrow \mathcal M$ for each $\ii \in \bz/n$ by
\begin{equation}
\begin{split}
\label{eq:efmdef}  
\tilde e_\ii^{\bf c}(M) &= 
\begin{cases}
0 & \text{ if } \varepsilon_\ii(M)=0   \\
A_{{\bf c}, \ii,k_e(M)-K} M & \text{ if } \varepsilon_\ii(M) > 0
\end{cases},\\
\tilde f_\ii^{\bf c}(M) &= 
\begin{cases}
0 & \text{ if } \varphi_\ii(M)=0 \\
A_{{\bf c}, \ii, k_f(M)}^{-1} M & \text{ if } \varphi_\ii(M) > 0
\end{cases}.
\end{split}
\end{equation}

We will actually need the following equivalent definition of $\tilde e_{\bar \imath}^{\bf c}$ and $\tilde f_\ii^{\bf c}$, which is described in a special case in \cite[\S 3]{monomial-partitions}. For each $\ii \in \bz/n$, let $S^M_\ii(M)$ be the string of brackets which contains a ``$($'' for every factor of $Y_{\ii, k}$ in $M$ and a ``$)$'' for every factor of $Y_{\ii, k}^{-1}$ in $M$, for all $k \in \bz$. These are ordered from left to right in decreasing order of $k$, as shown in Figure \ref{mono-rule}. Cancel brackets according to usual conventions (see \S\ref{ss:comb-cryst}), and set
\begin{equation} \label{ef-def2}
\begin{split}
  & \tilde{e}^{\bf c}_{\bar \imath}(M)  \hspace{-0.07cm} = \hspace{-0.07cm} 
  \begin{cases}
    0 & \mbox{} \hspace{-0.15cm}  \text{if there is no uncanceled ``)'' in $M$}, \\
   A_{{\bf c}; \ii,k-K}M & \mbox{} \hspace{-0.15cm}  \text{if the first  uncanceled ``)'' from the right comes from a factor $Y_{\ii, k}^{-1}$},
  \end{cases} \\
  & \tilde{f}^{\bf c}_\ii(M)  \hspace{-0.07cm} = \hspace{-0.07cm} 
  \begin{cases}
    0 &  \mbox{} \hspace{-0.15cm}  \text{if there is no uncanceled ``('' in $M$},  \\
   A_{{\bf c}; \ii, k}^{-1}M & \mbox{} \hspace{-0.15cm}  \text{if the first uncanceled ``('' from the left comes from a factor $Y_{\ii, k}$}.
  \end{cases}
\end{split}
\end{equation}
It is a straightforward exercise to see that these operators agree with \eqref{eq:efmdef}. 

\begin{figure}
\begin{center}
\setlength{\unitlength}{0.7cm}
\begin{picture}(11,1.5)

\put(0,0){$Y_{\bar 1,15}^{}$}

\put(1,0){$Y_{\bar 2,14}^{}$}
\put(2,0){$Y_{\bar 1,13}^{-2}$}
\put(3,0){$Y_{\bar 0,10}^{}$}
\put(4,0){$Y_{\bar 1,9}^{}$}
\put(5,0){$Y_{\bar 3,9}^{}$}

\put(6,0){$Y_{\bar 1,7}^{}$}
\put(7,0){$Y_{\bar 3,7}^{-1}$}
\put(8,0){$Y_{\bar 1,5}^{-1}$}
\put(9,0){$Y_{\bar 0,4}^{-1}$}
\put(10,0){$Y_{\bar 1,1}^{}$}

\put(0.1,1){$($}
\put(2.1,1){$))$}
\put(4.1,1){$($}

\put(6.1,1){$($}
\put(8.1,1){$)$}
\put(10.1,1){$($}

\end{picture}
\end{center}

\caption{Example: The string of brackets $S^M_{\overline 1}$  \label{mono-rule}}
\end{figure}

\begin{Theorem}[{\cite[Theorem 4.3]{Kashiwara:2001}}]
Assume $K=1$. Fix a monomial $D$ which is highest weight for the crystal structure and let $B^D$ be the subset generated by $D$. Then $B^D$ along with these operators is isomorphic to the crystal $B(\Lambda)$, where $\Lambda= \wt D$. 
\end{Theorem}

It is not true that all of $\mathcal{M}$ forms an $\asl_n$ crystal under these operations. This was noted by Nakajima in \cite[\S3]{Nakajima:2003} for some specific cases with $K=2$. However, here we will see that some elements of $\mathcal{M}$ do generate $\asl_n$ crystals, even for $K>1$. 

\begin{Remark}
Our results concern the case $K>1$, which is not considered by Kashiwara in \cite[\S4]{Kashiwara:2001}. However, Nakajima's original definition in \cite[\S3]{Nakajima:2003} essentially has $K=2$. Thus the higher $K$ cases were in fact part of the theory from the beginning. 
\end{Remark}
  
\subsection{Results on monomial crystals}
In this section we construct a map of crystals from certain instances of our crystals $B^\bx$ to certain monomial crystals. This is a generalization of \cite[Theorem 5.1]{monomial-partitions}, and the proof proceeds as in that paper. We then use this result to prove that some instances of the monomial crystal which have not previously been understood do in fact realize $B(\Lambda)$. 

Fix an integral, aligned slope datum $\bx$. We consider the crystal $B^\bx$ defined using the partial order (i) from \S\ref{ss:rational}. That is, boxes are ordered by $b \succ_\bx b'$ if and only if
\begin{itemize}
\item $h^{\bx}(b) > h^{\bx}(b')$ or
\item $h^{\bx}(b) = h^{\bx}(b')$ and $i_1>i_2$ or 
\item $h^{\bx}(b) = h^{\bx}(b')$ and $i_1=i_2$ and $k_1 > k_2.$
\end{itemize}
We consider the monomial crystal with parameters $c_{\ii,\ii+1}= \xi_\oOmega$ and $c_{\ii+1,\ii}=\xi_{\Omega}$ for all $\ii \in \bz/n$, so $K= \xi_\oOmega+\xi_{\Omega}$. Define a map $\Psi \colon B^\bx \cup \{ 0 \} \rightarrow \mathcal M \cup \{ 0 \}$ by $\Psi(0)=0$, and, for all $\bl \in B^\bx$,
\[
\Psi(\bl) := \prod_{a\in A(\bl)} Y_{\bar c(a), h^\bx(a)}  \prod_{r \in R(\bl)} Y_{\bar c(r), h^\bx(r)+K}^{-1}.
\]
Here $\bar c$, $A(\bl)$, and $R(\bl)$ are as in \S\ref{partitions-section}.

\begin{Theorem} \label{th:is-crystal-map}
For all $\bl \in B^\bx$, we have $\Psi(\tilde{e}_i^\bx(\bl)) = \tilde{e}_i^{\bf c}(\Psi(\bl))$, $\Psi(\tilde{f}_i^\bx(\bl)) = \tilde{f}_i^{\bf c}(\Psi(\bl))$, and $\wt(\Psi(\bl))= \wt(\bl)$. 
\end{Theorem}

The proof of Theorem \ref{th:is-crystal-map} will take most of the rest of this section. 

\begin{Lemma} \label{bv}
Let $\bl$ and $\bmu$ be multi-partitions such that $\bmu = \bl \sqcup b$ for some box $b$. Then $\Psi(\bmu) = A_{{\bf c}; \ii, h^\bx(b)}^{-1} \Psi(\bl),$ where $\ii= \bar c(b)$.
\end{Lemma}

\begin{proof} 
It is clear that $( A(\bl), R(\bl))$ differs from  $( A(\bmu), R(\bmu))$ in exactly the following four ways:
\begin{itemize}[leftmargin=10pt]
\item $b \in A_{\ii}(\bl) \setminus A_{\ii}(\bmu)$.

\item $b \in R_{\ii}(\bmu) \setminus R_\ii(\bl)$.

\item Either (i): $A_{\ii+ \bar 1}(\bmu) \setminus A_{\ii + \bar 1}(\bl)= b'$ and $R_{\ii + \bar 1}(\bl) =R_{\ii + \bar 1}(\bmu)$ for some box $b'$ with $h^\bx(b') = h^\bx(b)+\xi_\oOmega$, or (ii): $R_{\ii + \bar 1}(\bl) \setminus R_{\ii + \bar 1}(\bmu) = b'$ and $ A_{\ii + \bar 1}(\bl) =A_{\ii + \bar 1}(\bmu)$ for some box $b'$ with $h^\bx(b')= h^\bx(b)-\xi_\Omega$.

\item Either (i): $A_{\ii - \bar 1}(\bmu) \setminus A_{\ii - \bar 1}(\bl)= b''$ and $ R_{\ii - \bar 1}(\bl) =R_{\ii - \bar 1}(\bmu)$ for some box $b''$ with $h^\bx(b'') = h^\bx(b)+\xi_\Omega$, or (ii): $R_{\ii - \bar 1}(\bl) \setminus R_{\ii - \bar 1}(\bmu) = b''$ and $ A_{\ii - \bar 1}(\bl) =A_{\ii - \bar 1}(\bmu)$ for some box $b''$ with $h^\bx(b'')= h^\bx(b)-\xi_\oOmega$.
\end{itemize}
By the definition of $\Psi$, this implies that 
\begin{equation*}
\Psi(\bmu)= Y_{\ii, h^\bx(b)}^{-1}  Y_{\bar \imath, h^\bx(b)+\xi_\Omega+ \xi_{\oOmega}}^{-1}  Y_{\bar \imath+\bar 1, h^\bx(b) + \xi_\oOmega}^{}  Y_{\bar \imath-\bar 1, h^\bx(b) + \xi_{\Omega}}^{}  \Psi(\bl) = A_{{\bf c}; \ii, h^\bx(b)}^{-1} \Psi(\bl). \qedhere
\end{equation*}
\end{proof}

\begin{Lemma} \label{when-illegal} 
Fix $\bl \in B^\bx$. For any $a \in A_\ii(\bl)$ and $r \in R_\ii(\bl)$, we have $a \prec_\bx r$ if and only if $a \prec_\bx b$, where if $r = (k;i,j)$ we set $b=(k;i+1,j+1)$.  
\end{Lemma}

\begin{proof}
Any pair violating this will lead to an illegal triple (Lemma~\ref{lem:ar-illegal}), and hence is not in $B^\bx$. 
\end{proof}

\begin{proof}[Proof of Theorem \ref{th:is-crystal-map}]
Fix $\bl \in B^\bx$ and $\ii \in \bz/n$. Let $S^M_\ii(\Psi(\bl))$ denote the string of brackets used in \S\ref{ss:monom}. Let $S_\ii^\bx(\bl)$ denote the string of brackets used in \S\ref{ss:comb-cryst}, and define the height of a bracket in $S^\bx_\ii(\bl)$ to be $h^\bx(b)$ for the corresponding box.

It follows immediately from Lemma \ref{when-illegal} that, for each $k$, all ``$($'' in $S^\bx_\ii(\bl)$ of height $k+\xi_\Omega+ \xi_{\oOmega}$ are immediately to the left of all ``$)$'' of height $k$.  Let $T$ be the string of brackets obtained from $S^\bx_{\ii}(\bl)$ by, for each $k$, canceling as many ``$($'' of height $k+\xi_\Omega+ \xi_\oOmega$ with ``$)$'' of height $k$ as possible. One can use $T$ instead of $S^\bx_\ii(\bl)$ to calculate $\tilde e_\ii^\bx(\bl)$ and $\tilde f_\ii^\bx(\bl)$ without changing the result.

By the definition of $\Psi$, it is clear that 
\begin{enumerate}
\item The ``$($'' in $T$ of height $k+\xi_\Omega+\xi_\oOmega$ correspond exactly to the factors of $Y_{\ii, k+K}$ in $\Psi(\bl)$.
\item The ``$)$'' in $T$ of height $k$ correspond exactly to the factors of $Y_{\ii, k+K}^{-1}$ in $\Psi(\bl)$.
\end{enumerate}
Thus the brackets in $T$ correspond exactly to the brackets in $S^M_\ii(\Psi(\bl))$. Furthermore, Lemma \ref{when-illegal} implies that these brackets occur in the same order. The theorem then follows from Lemma \ref{bv} and the definitions of the crystal operators.
\end{proof}

\begin{Definition}
Fix $K$. We call a dominant monomial $M$ {\bf aligned} if
\begin{equation*}
\max \{ |k- k'| : k,k' \in \bz,  \text{ and for some } \ii, \ii' \in \bz/n \text{ with } y_{\ii, k}, y_{\ii', k'}  \neq 0\} < K. \qedhere
\end{equation*}
\end{Definition}

Theorem~\ref{th:is-crystal-map} has the following consequence, which shows that certain new instances of the monomial crystal do in fact still realize the crystal $B(\Lambda)$. This is related to \cite[Problem 2]{Kashiwara:2001}.

\begin{Corollary} 
Fix $\xi_{\Omega},\xi_{\oOmega} \in \bz_{>0}$. For all $\ii \in \bz/n$, let $c_{\ii,\ii+1}= \xi_\oOmega$ and $c_{\ii+1,\ii}=\xi_{\Omega}$. Let $M$ be a $(\xi_{\Omega}+\xi_{\oOmega})$-aligned dominant monomial. Then the component of the monomial crystal generated by $M$ under the operators $\tilde e_\ii$ and $\tilde f_\ii$ is a copy of $B(\Lambda)$ for $\Lambda= \wt(M).$ 
\end{Corollary}

\begin{proof}
One can easily find an aligned slope datum $\bx$ with $\xi_\Omega, \xi_\oOmega$ as specified, and such that $\Psi$ sends the empty multi-partition to $M$, and by Corollary \ref{cor:rat-crystal}, $B^\bx$ is a copy of the crystal $B(\Lambda)$. Thus the corollary follows immediately from Theorem \ref{th:is-crystal-map}.
\end{proof}

\section{Construction in terms of punctual quot schemes} \label{sec-quot}

Another reason the case of $\asl_n$ is special is that in this case $\mL(\bfv, W)$ can be realized using punctual quot schemes. We now briefly explain how our results translate into that language. Our references for this section are \cite{Nakajima:1999} and \cite{instantonlectures}.

\subsection{The punctual quot scheme}
Fix a finite dimensional $\bc$-vector space $W$. Choose coordinates on $\ba^2$ so that we can identify its coordinate ring $\cO_{\ba^2} = \bc[x,y]$. The punctual quot scheme $\Quot_0(W,m)$ is the moduli space of $\bc[x,y]$-submodules $K \subset \bc[x,y] \otimes W$ where the quotient $(\bc[x,y] \otimes W )/ K$ is $m$-dimensional and supported at the origin, i.e., annihilated by some power of the maximal ideal $(x,y)$. 

Fix a basis $\{ w_1, \ldots, w_\ell \}$ of $W$, a primitive $n^{\rm th}$ root of unity $\zeta$, and numbers $\alpha_j \in \bz/n$ for each $j \in \bz/n$. Consider the $\bz/n$ action on $ \bc[x,y] \otimes W$ defined by
\[
\overline k \cdot x = \zeta^{-k} x, \quad \overline k \cdot y = \zeta^{k} y, \quad \overline k \cdot w_j = \zeta^{\alpha_{j} k} w_j.
\]
From the results in \cite[\S 2]{Nakajima:1999} and \cite[\S 3]{instantonlectures}, we deduce that
\begin{equation} \label{eq:N1}
\Quot_0(W,m)^{\bz/n} \cong \coprod_{|\bfv|=m} \mL(\bfv, W).
\end{equation}
For each $K \in \Quot_0(W,m)^{\bz/n}$, the element $\overline 1 \in \bz/n$ defines an endomorphism of the vector space $(\bc[x,y] \otimes W)/K$, and the eigenvalues are $\zeta^k$ for $0 \leq k \leq n-1$. Fix $\bfv=(v_{\overline 0}, v_{\overline 1}, \ldots, v_{\overline n - \bar{1}})$ with $|\bfv|=m$. Define  
\[
\Quot_0(W,m)^{\bz/n}(\bfv)
\]
to be the subvariety of $\Quot_0(W,m)^{\bz/n}$ consisting of those $K$ where, for each $\overline k \in \bz/n$, the dimension of the $\zeta^k$-eigenspace of $(\bc[x,y] \otimes W)/K$ is $v_{\overline k}$. Then \eqref{eq:N1} can be refined to
\begin{equation} \label{eq:N2}
\Quot_0(W,m)^{\bz/n}(\bfv) \cong \mL(\bfv, W).
\end{equation}

\begin{Remark}
One can also realize $\mM(\bfv, W)$ in this picture.  Let $\MM(W,m)$ denote the framed moduli space of rank $\ell = \dim W$ torsion-free sheaves on $\bp^2$ whose second Chern class is $m$. For a certain $\bz/n$ action defined very similarly to the one above, one finds
\begin{equation*}
\MM(W,m)^{\bz/n} \cong \coprod_{|\bfv| = m} \mM(\bfv, W).
\end{equation*}
We can deduce from \cite[\S 2]{Nakajima:1999} and \cite[\S 3]{instantonlectures} that there is a natural $\bz/n$-equivariant embedding of $\Quot_0(W,m)$ into $\MM(W,m)$, which restricts to the usual embedding $\mL(\bfv, W) \hookrightarrow \mM(\bfv, W)$.
\end{Remark}

\subsection{Torus actions}
As in \S\ref{ss:torus-combinatorics}, let $T = \bc^{\ell+2}$, with chosen coordinates $(t_\Omega, t_\oOmega, t_1, \ldots, t_\ell)$. $T$ acts on $\bc[x,y] \otimes W$ where, for $t= (t_\Omega, t_\oOmega, t_1, \ldots, t_\ell)$,
\[
t \cdot x = t_\Omega^{-1} x, \;\; t \cdot y =t_\oOmega^{-1} y, \;\; t \cdot w_j = t_j^{-1} w_j.
\]
This induces an action of $T$ on $\Quot_0(W,m)$, and the fixed points are monomial submodules. They are parameterized by multi-partitions, where the fixed point corresponding to the multi-partition $\bl$ is  
\[
K_\bl := \text{span}_\bc \{ x^iy^j w_k \mid (i, j) \not \in \lambda^{(k)} \}.
\]

One can choose the isomorphism in \eqref{eq:N2} so that it intertwines this action with the action of $T$ on $\mL(\bfv, W)$ from \S\ref{ss:torus-combinatorics} and identifies fixed points corresponding to the same multi-partition. 

Fix an integral slope datum $\bx$ and consider $T_\bx \subset T$. The action of $T_\bx$ induces a Bia\l ynicki-Birula stratification of $\Quot_0(W,m)$ as follows: For any $K \in \Quot_0(W,m)$, $\lim_{t \to \infty} t \cdot_\bx K$ is the submodule of $\bc[x,y] \otimes W$ generated by the lowest degree (ordered by $\bx$) terms of elements of $K$. One can think of this as a ``reverse initial submodule'', and the strata are ``reverse Gr\"obner strata'', which we will denote by ${\rm RG}_\bl$. Note that this limit only makes sense since all $K \in \Quot_0(W,m)$ are supported at the origin. 

Each irreducible component $Z$ of $\Quot_0(W,m)^{\bz/n}(\bfv)$ decomposes as
\begin{equation*}
Z= \coprod_\bl (Z \cap {\rm RG}_\bl),
\end{equation*}
and exactly one of the sets $Z \cap {\rm RG}_\bl$ is dense in $Z$. This gives a map from the set of irreducible components of $\Quot_0(W,m)^{\bz/n}(\bfv)$ to multi-partitions.  Translated into this language, Theorem \ref{thm:to-partitions} says that, provided $\bx$ is aligned and sufficiently general, this map is injective, and its image consists of exactly the $\bx$-regular partitions.

\section{Questions}

We feel that the construction in \S\ref{sec:framework} should have more applications. We finish by formulating some questions related to this construction. 

\begin{Question}
Do interesting combinatorics arise when the construction from \S\ref{sec:framework} is carried out in other cases?
\end{Question}

\noindent 
In types other than $\mathfrak{sl}_n$ and $\asl_n$ the fixed-point components defined in \S\ref{sec:framework} are more complicated than just points, but perhaps it is still possible to index them combinatorially. Even the case of $\mathfrak{sl}_n$, where the fixed point components are just points, may lead to potentially new combinatorial realizations of $B(\lambda)$; one has freedom to choose the weights for the torus $T_W$, and this choice should lead to different realizations.

\begin{Question}
In type $\asl_n$, can one describe the situation combinatorially for more general $\bx$? 
\end{Question}

\noindent For instance, it would be natural to consider the case when $\bx$ is general but not aligned. By a combinatorial description we would mean a combinatorial characterization of $\im M_\bx$ and of the induced operators $e_i, f_i$ on $\im M_\bx$. By Theorem~\ref{thm:to-partitions}, $\im M_\bx$ consists only of $\bx$-regular multi-partitions, but one can easily find examples where it is a proper subset of these. Also, the construction in \S\ref{sec:framework} only requires $\xi_\Omega+ \xi_\oOmega >0$, not both individually to be positive. It may be interesting to understand our construction for this more general notion of slope datum. 

\begin{Question} \label{Qls}
Is there a natural crystal structure on a set of fixed point components larger than $\im M_\iota$? Or more geometrically, is there is natural crystal structure on the attracting sets of these components in the various $\mM(\bfv, W)$?
\end{Question}

One might naively hope for a natural crystal structure on all of $\coprod_\bfv \Irr \mathfrak{F}(\bfv, W)$, giving the crystal for a larger $\g$-representation, and which agrees with our the crystal structure on $\im M_\iota$. But, except for certain very specific choices of $\iota$, even in type $\asl_n$ the naive guess for this structure does not work (see Remark \ref{rem66}). However, one can sometimes find a crystal larger than $\im M_\iota$. Recall that, in the proof of Theorem \ref{thm:irrationalcrystal}, to show that a fixed point component was in $\im M_\iota$, we needed to show two things: 
\begin{enumerate}
\item that its attracting set has dimension $\dim(\mM(\bfv, W))/2$, and 
\item that its attracting set is contained in $\mL(\bfv,W)$.
\end{enumerate}
Define $R_\iota$ to be the set of fixed-point components that satisfy (i) but not necessarily (ii). At least in some cases, there is a natural crystal structure on $R_\iota$, as we now briefly describe. 

Choose a decomposition $W= W_1 \oplus W_2$. Consider $\iota$ defined by $\iota(z) = D$ where $D \in T_W$ is the diagonal matrix which acts on $W_1$ as the identity, and acts on $W_2$ as multiplication by $z^N$ for some large $N$ (this is roughly the action used by Nakajima in \cite{Nakajima:2001}). Deform $\iota$ by allowing non-trivial weights of $z$ embedding into $T_\oOmega$ and $T_s$, and changing the weights in $D$, but such that all changes to weights are much smaller than $N$. Then (at least in finite type; in other types one should be more careful with limits) this action preserves the tensor-product variety $\mE$ from \cite{Nakajima:2001}, so the map $M_\iota$ can be extended to a map $\widetilde M_\iota$ from $\Irr \mE$ to fixed-point components, and $R_\iota$ is exactly $\im \widetilde M_\iota$. Nakajima defines a crystal structure on $\Irr \mE$, so this can be pushed to a crystal structure on $R_\iota$. The result is the crystal for a tensor product of two highest weight representations. 

One interpretation of Question \ref{Qls} is to ask if Nakajima's tensor product variety can be generalized in the following sense: Choose a generic family $\iota^{(N)}$ as in Section~\ref{sec:framework}. For each $\bfv$, and large enough $N$, consider the subset $\mE(\bfv, W)$ of $\mM(\bfv, W)$ such that $\lim_{z \rightarrow \infty} \iota^{(N)}(z)$ exists. We conjecture that this is a subvariety of  $\mM(\bfv, W)$ of pure dimension $\dim(\mM(\bfv, W))/2$. Is there a natural crystal structure on $\Irr \mE(\bfv, W)$? If yes, is this the crystal of some $\g$ representation? We have not seriously considered these questions beyond the cases covered by Nakajima's previous work, but they seem like a natural course for further study.

\bibliographystyle{plain}

\begin{thebibliography}{Kas2}

\bibitem[Ber]{Berg:2010}
Chris Berg.
\newblock The ladder crystal.
\newblock {\it 	Electron. J. Combin.} {\bf 17} (2010), no.~1, Research Paper 97, 20 pp.
\newblock \arxiv{0901.3565v3}.

\bibitem[Bia]{BB} A.~Bia\l ynicki-Birula. Some theorems on actions of algebraic groups. {\it Ann. of Math. (2)} {\bf 98} (1973), 480--497.
  
\bibitem[CG]{CG:1997} Neil Chriss, Victor Ginzburg. 
{\it Representation theory and complex geometry}. 
Birkh\"auser, Boston, 1997.

\bibitem[Fay]{Fayers:2009} Matthew Fayers.  \newblock Partition models
  for the crystal of the basic ${U}_q(\widehat{\frak{sl}}_n)$-module.
  \newblock {\it J. Algebraic Combin.} {\bf 32} (2010) no.~3, 339--370.  \newblock
  \arxiv{0906.4129v3}.

\bibitem[FS]{frenkelsavage} Igor B. Frenkel, Alistair Savage. 
Bases of representations of type $A$ affine Lie algebras via quiver varieties and  statistical mechanics. {\it Int. Math. Res. Not.} {\bf 28} (2003), 1521--1547, 	\arxiv{math/0211452v2}.

\bibitem[FM]{fujiiminabe} 
Shigeyuki Fujii, Satoshi Minabe.
A combinatorial study on quiver varieties. 
\arxiv{math/0510455v2}.

\bibitem[Hai]{haiman} Mark Haiman. 
$t,q$-Catalan numbers and the Hilbert scheme.
Selected papers in honor of Adriano Garsia (Taormina, 1994),
{\it Discrete Math.} {\bf 193} (1998), no.~1-3, 201--224. 

\bibitem[Har]{harris} Joe Harris.
{\it Algebraic geometry, a first course}.
Graduate Texts in Math. {\bf 133}, Springer-Verlag, 1992. 

\bibitem[HK]{Hong&Kang:2000}
Jin Hong, Seok-Jin Kang.
\newblock {\it Introduction to quantum groups and crystal bases}, volume~42 of
  {\it Graduate Studies in Mathematics}.
\newblock American Mathematical Society, Providence, RI, 2002.

\bibitem[Kas1]{Kashiwara:1995}
M.~Kashiwara.
\newblock On crystal bases. {\it Representations of groups (Banff, AB, 1994)}, 155--197, CMS Conf. Proc., {\bf 16}, Amer. Math. Soc., Providence, RI, 1995.

\bibitem[Kas2]{Kashiwara:2001}
M.~Kashiwara.
\newblock Realizations of crystals.
\newblock In {\it Combinatorial and geometric representation theory ({S}eoul,
  2001)}, volume 325 of {\it Contemp. Math.}, pages 133--139. Amer. Math. Soc.,
  Providence, RI, 2003.
\newblock \arxiv{math/0202268v1}.

\bibitem[KS]{kashiwarasaito} Masaki Kashiwara, Yoshihisa Saito.
  Geometric construction of crystal bases. {\it Duke Math. J.} {\bf
    89} (1997), no.~1, 9--36.
 
\bibitem[Lus]{Lusztig:1991} George Lusztig. 
Quivers, perverse sheaves, and quantized enveloping algebras. 
{\it J. Amer. Math. Soc.}
{\bf 4} (1991), no.~2, 365--421. 
    
 \bibitem[MW]{MW} Jerrold Marsden, Alan Weinstein. 
Reduction of symplectic manifolds with symmetry. {\it Rep. Mathematical Phys.} {\bf 5}
  (1974), no.~1, 121--130.

\bibitem[MM]{MM:1990}
Kailash Misra, Tetsuji Miwa.
\newblock Crystal base for the basic representation of {$U_q(\widehat{\mathfrak{sl}}_n)$}.
\newblock {\it Comm. Math. Phys.}, {\bf 134} (1990), no.~1, 79--88.

\bibitem[Mum]{mumford}
David Mumford. 
{\it The red book of varieties and schemes}, second, expanded edition, with contributions by Enrico Arbarello. 
Lecture Notes in Mathematics {\bf 1358}, Springer-Verlag, Berlin, 1999.

\bibitem[Nak1]{Nakajima:1994} Hiraku Nakajima. 
Instantons on {ALE} spaces, quiver varieties, and {K}ac-{M}oody algebras. {\it Duke Math. J.} {\bf 76} (1994), no.~2, 365--416.

\bibitem[Nak2]{Nakajima:1994b} Hiraku Nakajima. 
Homology of moduli spaces of instantons on ALE spaces. I. {\it J. Differential Geom.} {\bf 40} (1994), no.~1, 105--127. 

\bibitem[Nak3]{Nakajima:1998} Hiraku Nakajima. 
Quiver varieties and {K}ac--{M}oody algebras. 
{\it Duke Math. J.} {\bf 91} (1998), no.~3, 515--560.

\bibitem[Nak4]{Nakajima:1999} Hiraku Nakajima. {\it Lectures on {H}ilbert schemes of
    points on surfaces}, volume~18 of {\it University Lecture Series},
  American Mathematical Society, Providence, RI, 1999.

\bibitem[Nak5]{Nakajima:2001} Hiraku Nakajima. Quiver varieties and tensor products. {\it Invent. Math.} {\bf 146} (2001), no.~2, 399--449. \arxiv{math/0103008v2}.

\bibitem[Nak6]{Nakajima:2003}
Hiraku Nakajima.
\newblock {$t$}-analogs of {$q$}-characters of quantum affine algebras of type
  {$A\sb n,D\sb n$}.
\newblock In {\it Combinatorial and geometric representation theory ({S}eoul,
  2001)}, volume 325 of {\it Contemp. Math.}, pages 141--160. Amer. Math. Soc.,
  Providence, RI, 2003.
\newblock \arxiv{math/0204184v1}.

\bibitem[NY1]{instantonlectures} Hiraku Nakajima, K\=ota Yoshioka. 
Lectures on instanton counting. {\it Algebraic structures and moduli spaces},
  31--101, CRM Proc. Lecture Notes, 38, Amer. Math. Soc., Providence,
  RI, 2004. \arxiv{math/0311058v1}.

\bibitem[NY2]{instantonI} Hiraku Nakajima, K\=ota Yoshioka. Instanton counting on
  blowup. I. 4-dimensional pure gauge theory. 
{\it Invent. Math.} {\bf 162} (2005), no.~2, 313--355. \arxiv{math/0306198v2}.

\bibitem[Sai]{Saito:2002}
Yoshihisa Saito.
\newblock Crystal bases and quiver varieties.
\newblock {\it Math. Ann.} {\bf 324} (2002), no.~4, 675--688. 
\newblock \arxiv{math/0111232v1}. 

\bibitem[Sav]{Savage:2006}
Alistair Savage.
\newblock Geometric and combinatorial realizations of crystal graphs.
\newblock {\it Algebr. Represent. Theory} {\bf 9} (2006), no.~2, 161--199.
\newblock \arxiv{math/0310314v4}.

\bibitem[Tin]{monomial-partitions} Peter Tingley. Monomial crystals and partition crystals. 
{\it SIGMA Symmetry Integrability Geom. Methods Appl.} {\bf 6} (2010), Paper 035, 8pp. \arxiv{0909.2242v3}.

\end{thebibliography}

\end{document}